\documentclass[a4paper]{amsart}
\usepackage[utf8]{inputenc}
\usepackage[T1]{fontenc}
\usepackage{lmodern}
\usepackage{microtype}
\usepackage{mathrsfs}		% \mathscr  
\usepackage{csquotes}		% recommended for biblatex with babel
\usepackage{mathtools}		% paired delimiters
\usepackage{enumitem}		% customized lists
\usepackage[ngerman,american]{babel}	% language features
\usepackage{tikz-cd}					% commutative diagrams
\usepackage[
backend=biber,
giveninits=true,
isbn=false,
doi=false,
url=false
]{biblatex}
\usepackage{hyperref}
\usepackage[
nameinlink,
noabbrev
]{cleveref}

\allowdisplaybreaks
\newlist{enumtheo}{enumerate}{1}
\newlist{enumstep}{enumerate}{1}
\setlist[enumerate,itemize]{wide}
\setlist[enumtheo,enumstep]{wide,label=(\roman*),font=\upshape}
\setlist[enumerate,1]{label=\alph*),font=\upshape}
\swapnumbers
\theoremstyle{plain}
\newtheorem{theo}[subsection]{Theorem}%[section]
\newtheorem{lemm}[subsection]{Lemma}
\newtheorem{prop}[subsection]{Proposition}
\newtheorem{coro}[subsection]{Corollary}
\theoremstyle{definition}
\newtheorem{defi}[subsection]{Definition}

\newtheorem{exam}[subsection]{Example}
\theoremstyle{remark}
\newtheorem{rema}[subsection]{Remark}
\numberwithin{equation}{subsection}
\creflabelformat{equation}{#2#1#3}
\crefname{rema}{remark}{remarks}
\crefname{defi}{definition}{definitions}
\crefname{theo}{theorem}{theorems}
\crefname{prop}{proposition}{propositions}
\crefname{coro}{corollary}{corollaries}
\crefname{enumtheoi}{item}{items}
\crefname{enumstepi}{step}{steps}
\crefformat{subsection}{#2(#1)#3}
\crefname{diag}{diagram}{diagrams}
\crefname{sequ}{sequence}{sequences}
\newcommand{\NN}{\mathbb{N}}				% set of natural number
\newcommand{\ZZ}{\mathbb{Z}}				% set/ring of integers
\newcommand{\CC}{\mathbb{C}}				% set/ring of complex numbers
\newcommand{\HH}{\mathbb{H}}				% quaternions
\newcommand{\Oka}{\mathscr{O}}				% Oka-O, structure sheaf
\newcommand{\pt}{\mathsf{pt}}		% terminal complex space
\newcommand{\cC}{\mathcal{C}}

\newcommand{\cF}{\mathcal{F}}

\newcommand{\cK}{\mathcal{K}}

\newcommand{\cX}{\mathcal{X}}
\newcommand{\cY}{\mathcal{Y}}

\newcommand{\sF}{\mathscr{F}}
\newcommand{\sG}{\mathscr{G}}
\newcommand{\sH}{\mathscr{H}}
\newcommand{\sL}{\mathscr{L}}
\newcommand{\pdom}[1]{\mathsf{D}_{#1}}			% period domain
\newcommand{\id}[1]{\mathsf{id}_{#1}}		% identity map
\newcommand{\pr}[1]{\mathsf{pr}_{#1}}		% cartesian projection
\newcommand{\PP}[1]{\mathbb{P}(#1)}			% projective space of lines
\newcommand{\CP}[1]{\mathbb{P}^{#1}}		% complex projective space
\newcommand{\Tsp}[2]{\mathsf{T}_{#2}#1}		% tangent space
\newcommand{\Nb}[2]{\mathscr{N}_{#1/#2}}	% normal bundle
\newcommand{\csh}[2]{\underline{#1}_{#2}}	% constant sheaf on space
\newcommand{\hodge}[3]{\mathsf{h}^{#1,#2}(#3)}		% Hodge number
\newcommand{\Hsh}[3]{\mathsf{H}^{#1}(#2;#3)}		% sheaf cohomology
\newcommand{\hdim}[3]{\mathsf{h}^{#1}(#2;#3)}		% sheaf cohomology dimension
\newcommand{\Hdg}[3]{\mathsf{H}^{#1,#2}(#3)}		% Hodge subspace
\newcommand{\Rsh}[3]{\mathsf{R}^{#1}#2_*{#3}}		% Higher direct image sheaf
\newcommand{\Rshp}[3]{\mathsf{R}^{#1}#2_*(#3)}		% same with parentheses
\newcommand{\rk}{\operatorname{\mathsf{rk}}}		% rank (of a quadratic form)
\newcommand{\Aut}{\operatorname{\mathsf{Aut}}}		% Aut-set
\newcommand{\Pic}{\operatorname{\mathsf{Pic}}}		% Picard group
\newcommand{\Sp}{\operatorname{\mathsf{Sp}}}		% hyperunitary group
\newcommand{\dou}{\operatorname{\mathsf{Dou}}}		% Douady space
\newcommand{\sHom}{\operatorname{\sH\!\mathit{om}}}	% sheaf Hom
\renewcommand{\deg}{\operatorname{\mathsf{deg}}}	% degree (map or family)
\DeclarePairedDelimiter\paren{(}{)}					% ordinary parentheses
\DeclarePairedDelimiter\set{\{}{\}}					% set brackets
\DeclarePairedDelimiter\abs{\lvert}{\rvert}			% absolute value
\DeclarePairedDelimiterX\bil[2]{\langle}{\rangle}{#1,#2}
\DeclarePairedDelimiterX\range[2]{\{}{\}}{#1,\dots,#2}	% interval of integers
\DeclarePairedDelimiterX\setb[2]\lbrace\rbrace{#1 \;\delimsize\vert\; #2}	% set builder notation
\renewcommand{\subset}{\subseteq}
\renewcommand{\epsilon}{\varepsilon}
\newcommand{\conj}{\overline}			% complex conjugation
\newcommand{\closure}{\overline}		% topological closure
\newcommand{\rest}[2]{\left.{#1}\right\vert_{#2}}	% restriction
\newcommand{\from}{\colon}
\newcommand{\isom}{\cong}
\newcommand{\inv}{^{-1}}

\newcommand{\x}{\times}
\newcommand{\inj}{\hookrightarrow}

\newcommand*{\defeq}{\coloneqq}

\renewcommand{\tilde}{\widetilde}
\hypersetup{
breaklinks		= true,
pdftitle		= {Unobstructedness of hyperkähler twistor spaces},
pdfauthor		= {Ana-Maria Brecan, Tim Kirschner, Martin Schwald},
pdfkeywords		= {},
pdfprintscaling	= None
}
\author{Ana-Maria Brecan}
\address{Fakultät für Mathematik, Physik und Informatik\\ Universität Bayreuth}
\email{ana-maria.brecan@uni-bayreuth.de}
\urladdr{\url{http://www.komplexe-analysis.uni-bayreuth.de/de/team/Brecan_Ana-Maria/}}
\author{Tim Kirschner}
\address{Fakultät für Mathematik\\ Universität Duisburg-Essen}
\email{tim.kirschner@uni-due.de}
\urladdr{\url{http://www.esaga.uni-due.de/tim.kirschner/}}
\author{Martin Schwald}
\address{Fakultät für Mathematik\\ Universität Duisburg-Essen}
\email{martin.schwald@uni-due.de}
\urladdr{\url{http://www.esaga.uni-due.de/martin.schwald/}}

\begin{document}
\title{Unobstructedness of hyperkähler twistor spaces}
\date\today
\thanks{Tim Kirschner and Martin Schwald are supported by SFB/Transregio 45 of the DFG}
\begin{abstract}
A family of irreducible holomorphic symplectic (ihs) manifolds over the complex projective line has unobstructed deformations if its period map is an embedding. This applies in particular to twistor spaces of ihs manifolds. Moreover, a family of ihs manifolds over a subspace of the period domain extends to a universal family over an open neighborhood in the period domain.
\end{abstract}
\maketitle
\tableofcontents

\section{Introduction}

\subsection{}
Unobstructedness is one of the fundamental deformation-theoretic properties that a compact complex manifold $X$ can enjoy. The term goes back to Kodaira and Spencer \cite{KS58} who call elements $\xi \in \Hsh1X{\Theta_X}$ “obstructed” when $[\xi,\xi]\ne0$ in the cohomology group $\Hsh2X{\Theta_X}$, where $\Theta_X$ denotes the sheaf of holomorphic vector fields on $X$ and the bracket combines the sheaf-cohomological cup product and the ordinary Lie bracket of vector fields.

In modern terminology we say that $X$ is \emph{unobstructed} or has \emph{unobstructed deformations} when $X$ possesses a semi-universal, also known as “miniversal,” deformation over a smooth pointed complex space; in other words, $X$ possesses a smooth local moduli space. The old and new notions of (un)obstructedness are related by the following observation: Given a deformation of $X$ over a smooth pointed complex space $(B,b)$ and an element $\xi$ in the image of the associated Kodaira--Spencer map $\Tsp Bb \to \Hsh1X{\Theta_X}$, we know that $[\xi,\xi] = 0$. Since the Kodaira--Spencer map of a complete deformation maps onto $\Hsh1X{\Theta_X}$, we see that no obstructed elements can exist if $X$ is unobstructed. Conversely, by a theorem of Kodaira, Spencer, and Nirenberg, $X$ is unobstructed when the group $\Hsh2X{\Theta_X}$ is trivial \cite{KNS58}.

\subsection{}
In this note we investigate the local deformation theory of compact complex manifolds $X$ that admit a holomorphic submersion $f \colon X \to \CP1$ to the complex projective line such that for every point $t \in \CP1$ the fiber $X_t = f^{-1}(t)$ is an irreducible holomorphic symplectic manifold. We say that $(X,\CP1,f)$ is a family of ihs manifolds under these circumstances.

The main motivation for considering such $X$ is \emph{twistor space}. Recall that when $(M,g)$ is a connected Riemannian manifold whose holonomy group at a point $x \in M$ is, under a suitable identification of inner product spaces $\Tsp Mx \cong \HH^n$, equal to the quaternionic unitary group $\Sp(n)$, then the space $C$ of Kählerian complex structures on $(M,g)$ is diffeomorphic to $\CP1$. The twistor space of $(M,g)$ bundles the complex structures contained in $C$ into a $(2n+1)$-dimensional complex manifold $Z$ such that the underlying differentiable manifold of $Z$ is $M\times C$ and the projection map $\pr2 \colon Z \to C$ is holomorphic for a choice of complex structure on $C$ \cite{HKLR}. If $M$ is compact and $n>0$, we recover a family of ihs manifolds.

\begin{theo}[Unobstructedness theorem]\label{unobs thm}
Let $(X,\CP1,f)$ be a family of ihs manifolds whose period map $h\colon\CP1 \to \pdom\Lambda$, with respect to a $\Lambda$-marking $\mu$, is an embedding. Then $X$ has unobstructed deformations.

When $r$ is the second Betti number of a fiber of $f$ and $d=-\deg (f_*\Omega^2_{X/\CP1})$, then $d\ge2$ and
\begin{equation*}
\dim_\CC\Hsh1X{\Theta_X} = (r-2)(d+1)-3.
\end{equation*}
Moreover, the space $\Hsh0X{\Theta_X}$ is trivial.
\end{theo}

As a corollary we obtain that twistor spaces of compact irreducible hyperkähler manifolds are unobstructed. In fact, in the twistor case we see that $d=2$, which implies that the condition on the period map in \cref{unobs thm} is automatic. Compare \cref{twistor example,closed embedding}.

\subsection{}
The most popular unobstructedness criteria cannot be applied to prove \cref{unobs thm}---for example, the group $\Hsh2X{\Theta_X}$ is typically nontrivial (\cref{tg cohom thm}) and $X$ is neither Kählerian nor has trivial canonical bundle (\cref{canonical}). Thus the criteria of Kodaira--Spencer--Nirenberg and Tian--Todorov \cites{Tia87,Tod89} fail, respectively.

Our proof of \cref{unobs thm} follows a hands-on approach. The first key ingredient is that, essentially by virtue of Kodaira's theorem on the stability of fiber structures \cite{Kod63}, every deformation of $X$ induces a deformation of $h(\CP1)$ in $\pdom{\Lambda}$. Hence we obtain a description of the local moduli space of $X$ as a germ of the Douady space of $\pdom\Lambda$.
This leads us to think of \cref{unobs thm} not merely as an abstract deformation-theoretic statement, but as the primary step towards a moduli theory for families of ihs manifolds over $\CP1$. We want to emphasize this advantage of our approach over, for instance, the techniques of Ran who has obtained similar results \cite[116--117]{Ran92}. Ran's abstract functorial approach, however, can never yield a tangible description of the moduli space of interest.

\subsection{}
As a second key ingredient for \cref{unobs thm} we prove an extension theorem for $\Lambda$-marked families of ihs manifolds. The extension theorem implies that every deformation of $h(\CP1)$ in $\pdom\Lambda$ lifts to a deformation of $X$. We hope that this theorem meets interest aside from its use in this paper: While we know that a universal family on the moduli space of $\Lambda$-marked ihs manifolds does not exist \cite[][Remark 4.4]{Huy12}, we can still produce universal families on large open subspaces.

\begin{theo}[Extension theorem]\label{ihs extension thm}
Let $\cF$ be a $\Lambda$-marked family of ihs manifolds over a complex space $S$ such that the period map $h \colon S \to \pdom\Lambda$ of $\cF$ is an embedding. Then there exists a $\Lambda$-marked family of ihs manifolds $\tilde\cF$ over an open subspace $U$ of $\pdom\Lambda$ such that
\begin{enumerate}
\item $h(S)$ is contained in $U$,
\item there exists a morphism of $\Lambda$-marked families $\phi \colon \cF \to \tilde\cF$ over the induced map $h \colon S \to U$, and
\item the period map of $\tilde\cF$ is the canonical injection $U \inj \pdom\Lambda$.
\end{enumerate}
\end{theo}

\subsection{Outline of the paper}
In \cref{conventions sec} we compile, for easy reference, a list of terminology and notation that we employ throughout the paper. We begin a systematic investigation of families of ihs manifolds over smooth rational curves in \cref{systematic sec}. The main part of the proof of \cref{unobs thm} is carried out in \cref{unobs sec}. Even though they enter into \cref{unobs sec}, we postpone certain sheaf-cohomological computations to \cref{tg cohom sec} and the fairly technical proof of \cref{ihs extension thm} to the end of our paper---namely, to \cref{universal map sec,extension sec}.

\subsection*{Acknowledgments}
Daniel Greb deserves considerable thanks for he has been fostering our collaboration since day one.
We thank Daniel Huybrechts for answering several of our questions in the early stages of the project.
Ana-Maria Brecan extends her gratitude to Alan Huckleberry and Hans Franzen for their unceasing support.
Tim Kirschner thanks Jun-Muk Hwang for the opportunity to work on this paper at the Korea Institute for Advanced Study.

\section{Conventions, terminology, notation}\label{conventions sec}

\subsection{}\label{spaces sec}
In our definitions of complex spaces, holomorphic maps, subspaces, etc. we follow Gerd Fischer \cite[9--10]{Fischer}. In particular we assume that the underlying topological space of a complex space is Hausdorff.

\subsection{}\label{families sec}
A \emph{family of compact complex manifolds} is a triple $\cF=(X,S,f)$ where $X$ and $S$ are complex spaces and $f \colon X\to S$ is a proper holomorphic submersion. The complex space $S$ is called the \emph{base space} of the family $\cF$. When $S$ is the base space of $\cF$, we say that $\cF$ is a family \emph{over} $S$. 

Given two families $\cF = (X,S,f)$ and ${\cF' = (X',S',f')}$ of compact complex manifolds, a \emph{morphism of families} from $\cF'$ to $\cF$ is a pair $\phi = (g,h)$ such that
\begin{equation}\label[diag]{mor of families}
\begin{tikzcd}
X' \rar{f} \dar{f'} & X \dar{f} \\ S' \rar{h} & S
\end{tikzcd}
\end{equation}
is a Cartesian square of complex spaces and holomorphic maps. We write this as $\phi \colon \cF' \to \cF$, and we say that $\phi$ is a morphism \emph{over} $h$. Note that $\phi$ is an isomorphism of families if $h$ is an isomorphism of complex spaces. We call a morphism of families over $\id S\colon S\to S$ an $S$-\emph{morphism}, or $S$-\emph{isomorphism}, \emph{of families}.

\subsection{}\label{pullback sec}
Let $b \colon T \to S$ be a holomorphic map. Then for every family of compact complex manifolds $\cF = (X,S,f)$ over $S$, the \emph{pullback} of $\cF$ by $b$ is the triple
\begin{equation}\label{pullback triple}
b^*(\cF) \coloneqq (X \times_S T,T,p_2),
\end{equation}
where $X\times_ST$ denotes the fiber product of complex spaces over the holomorphic maps $f\colon X\to S$ and $b\colon T \to S$ and where $p_i$, for $i \in \set{1,2}$, denotes the $i$th projection map of this fiber product. Observe that $b^*(\cF)$ is a family of compact complex manifolds over $T$ and that $\eta \coloneqq (p_1,b) \colon b^*(\cF) \to \cF$ is a morphism of families over $b$.

When $\phi \colon \cF' \to \cF$ is an $S$-morphism of families, we denote by $b^*(\phi)$ the unique $T$-morphism of families $b^*(\cF')\to b^*(\cF)$ such that $\eta \circ b^*(\phi) = \phi \circ \eta'$ where $\eta'$ is the canonical morphism of families $b^*(\cF') \to \cF'$ over $b$.

\subsection{}\label{restriction sec}
When the holomorphic map $b \colon T\to S$ in \cref{pullback sec} is the canonical injection of a complex subspace, we write $\cF_T$ and $\phi_T$ for $b^*(\cF)$ and $b^*(\phi)$, respectively. Observe that in this case the fiber product $X\times_ST$ appearing in \cref{pullback triple} is nothing but the inverse image $f^{-1}(T)$ of the complex subspace $T \subset S$ under the holomorphic map $f \colon X\to S$ \cite[23]{Fischer}. Moreover, the first and second projections of the fiber product correspond to the canonical injection of $f^{-1}(T)$ and the restriction of $f$, respectively.

\subsection{}\label{fiber sec}
When $f\colon X\to S$ is a holomorphic map and $s \in S$ is a point, we use the standard notation $X_s$ for the complex analytic fiber of $f$ over $s$. When $\cF = (X,S,f)$ is a family of compact complex manifolds, we write $\cF(s)$ as a synonym for $X_s$ and call this the \emph{fiber} of $\cF$ over $s$, too.

When $\phi \colon \cF' \to \cF$ is a morphism of families over a holomorphic map $h \colon S' \to S$, then for every point $s \in S'$ we let $\phi(s)$ denote the induced holomorphic map $\cF'(s) \to \cF(h(s))$, which is a biholomorphism.

\subsection{}\label{ihs sec}
An \emph{irreducible holomorphic symplectic manifold}---for short, \emph{ihs manifold}---is a simply connected, compact, Kählerian, holomorphic symplectic complex manifold $X$ satisfying $\hdim0X{\Omega^2_X} = 1$ \cite[cf.][763--764]{Bea83}.

The cohomology group $\Hsh2X\ZZ$ of an ihs manifold $X$ is naturally, by virtue of a rescaling of the Beauville--Bogomolov form of $X$, a \emph{lattice}---that is, a free abelian group of finite rank endowed with a symmetric integral bilinear form \cite[Théorème 5]{Bea83}. Indeed there is a unique such rescaling by a minimal, strictly positive real number. When $\Lambda$ is a lattice, a $\Lambda$-\emph{marking} of $X$ is a lattice isomorphism $\mu \colon \Hsh2X\ZZ \to \Lambda$.

\subsection{}\label{ihs family sec}
A \emph{family of ihs manifolds} is a family of compact complex manifolds $\cF$ such that for every point $s$ of the base space of $\cF$ the fiber $\cF(s)$ is an ihs manifold.

When $\Lambda$ is a lattice, a $\Lambda$-\emph{marking} of a family of ihs manifolds $\cF = (X,S,f)$ is an isomorphism of sheaves of abelian groups $\mu \colon \Rsh2f{\csh\ZZ X} \to \csh\Lambda S$ such that, for every point $s \in S$, the induced map $\mu_s \colon \Hsh2{X_s}\ZZ \to \Lambda$ is a $\Lambda$-marking of the fiber $X_s = \cF(s)$ of $\cF$. A $\Lambda$-\emph{marked family (of ihs manifolds)} is a pair $(\cF,\mu)$ where $\cF$ is a family of ihs manifolds and $\mu$ is a $\Lambda$-marking of $\cF$. When $\cF$ is a family over $S$, we say that $(\cF,\mu)$ is a $\Lambda$-marked family \emph{over} $S$. We might occasionally drop the reference to $\Lambda$ in our terminology.

\subsection{}\label{marked mor sec}
Let $\Lambda$ be a lattice and $(\cF,\mu) = (X,S,f,\mu)$ and $(\cF',\mu') = (X',S',f',\mu')$ be $\Lambda$-marked families of ihs manifolds. Then a \emph{morphism of $\Lambda$-marked families} from $(\cF',\mu')$ to $(\cF,\mu)$ is a morphism of families $\phi = (g,h) \colon \cF' \to \cF$ such that the following diagram of sheaves of abelian groups on $S'$ commutes:
\begin{equation}\label[diag]{marked mor diagram}
\begin{tikzcd}
h^{-1}(\Rsh2f{\csh\ZZ X}) \rar{\alpha} \dar{h^{-1}(\mu)} & \Rsh2{f'}{\csh\ZZ{X'}} \dar{\mu'} \\
h^{-1}(\csh\Lambda S) \rar{\beta} & \csh\Lambda{S'}
\end{tikzcd}
\end{equation}
In \cref{marked mor diagram}, $\alpha$ denotes the topological base change map associated to \cref{mor of families} and the constant sheaves of abelian groups with value $\ZZ$. Analogously $\beta$ denotes the canonical $h$-map between the constant sheaves of abelian groups with value $\Lambda$. As usual we write $\phi \colon (\cF',\mu') \to (\cF,\mu)$ for the fact that $\phi$ is a morphism from $(\cF',\mu')$ to $(\cF,\mu)$.

\begin{rema}\label{induced marking}
Let $\cF$ and $\cF'$ be families of ihs manifolds, $\phi \colon \cF' \to \cF$ be a morphism of families, $\Lambda$ be a lattice, and $\mu$ be a $\Lambda$-marking of $\cF$. Then there exists a unique $\Lambda$-marking $\mu'$ of $\cF'$ so that $\phi \colon (\cF',\mu') \to (\cF,\mu)$ is a morphism of $\Lambda$-marked families.
\end{rema}
\begin{proof}
Write $\cF = (X,S,f)$ and $\cF' = (X',S',f')$. Then since $f \colon X\to S$ is a proper holomorphic map and since \cref{mor of families} is a Cartesian square of complex spaces, the topological base change map $\alpha$ in \cref{marked mor diagram} is an isomorphism of sheaves of abelian groups on $S'$. This proves the uniqueness. To see the existence define $\mu'$ as the composition of $\alpha^{-1}$, $h^{-1}(\mu)$, and $\beta$. Then $\mu'$ is an isomorphism of sheaves of abelian groups, for $\alpha$, $\beta$, and $\mu$ are. Moreover, for every point $s \in S'$, we see that $\mu'_s \circ \phi(s)^* = \mu_{h(s)}$ where $\mu_{h(s)} \colon \Hsh2{X_{h(s)}}\ZZ \to \Lambda$ and $\mu'_s \colon \Hsh2{X'_s}\ZZ \to \Lambda$ denote the maps induced by $\mu$ and $\mu'$, respectively, and where $\phi(s)^*$ denotes the map that $\phi(s) \colon X'_s \to X_{h(s)}$ induces on the second cohomology with values in $\ZZ$. Since $\mu_{h(s)}$ and $\phi(s)^*$ are isomorphisms of lattices, this proves that $\mu'_s$ is a $\Lambda$-marking of $X'_s$.
\end{proof}

\subsection{}\label{marked pullback sec}
By virtue of \cref{induced marking} we are able to adapt the language and notation of \cref{pullback sec,restriction sec} for marked families.
Indeed when $b \colon T\to S$ is a holomorphic map and $\cF$ is a family of ihs manifolds over $S$, then $b^*(\cF)$ is a family of ihs manifolds over $T$. Moreover, when $\Lambda$ is a lattice and $\mu$ is a $\Lambda$-marking of $\cF$, there exists a unique marking $\nu$ of $b^*(\cF)$ so that the canonical morphism of families $\eta \colon b^*(\cF) \to \cF$ is a morphism of $\Lambda$-marked families from $(b^*(\cF),\nu)$ to $(\cF,\mu)$. Accordingly we define the \emph{pullback} of the $\Lambda$-marked family $(\cF,\mu)$ by $b$ as $b^*(\cF,\mu) \coloneqq (b^*(\cF),\nu)$. Observe that when $\phi \colon (\cF',\mu') \to (\cF,\mu)$ is an $S$-morphism of $\Lambda$-marked families, then the $T$-morphism of families $b^*(\phi) \colon b^*(\cF') \to b^*(\cF)$ defined in \cref{pullback sec} is a morphism of $\Lambda$-marked families from $b^*(\cF',\mu')$ to $b^*(\cF,\mu)$.

\subsection{}\label{pdom sec}
Given a lattice $\Lambda$ of rank $r\ge 3$ and signature $(3,r-3)$ we let $\pdom\Lambda$ denote the \emph{period domain} associated to $\Lambda$; that is, $\pdom\Lambda$ is the complex subspace of the projective space of lines $\PP{\Lambda_\CC}$ induced on the locally closed analytic subset
\[
\setb{\CC x}{x \in \Lambda_\CC \setminus \set0,\; xx=0,\; x\conj x>0}
\]
where $\Lambda_\CC \coloneqq \CC \otimes_\ZZ \Lambda$. Note that the complex vector space $\Lambda_\CC$ is naturally endowed, for one, with a symmetric complex bilinear form written $(x,y) \mapsto xy$, which extends the bilinear form of $\Lambda$, and, for another, with a real structure written $x \mapsto \conj x$. Note furthermore that the equation $xx = 0$ defines a nondegenerate (i.e., smooth) quadric $Q$ in the projective space $\PP{\Lambda_\CC}$. The period domain $\pdom\Lambda$ can thus be viewed as an open complex submanifold of $Q$.

\subsection{}\label{pm sec}
Let $\cF = (X,S,f)$ be a family of ihs manifolds. Then the sheaf of $\Oka_S$-modules $f_*\Omega^2_{X/S}$ is locally free of rank $1$. When the complex space $S$ is reduced, this is a direct consequence of Grauert's base change theorem \cite[64]{EinTheorem}. The statement remains true though for arbitrary $S$. One way to see this is to invoke the unobstructedness of ihs manifolds, see \cref{local Torelli}, by which the family $\cF$ is---at least locally at every point of $S$---isomorphic to the pullback of a family of ihs manifolds over a smooth complex space. For the family over the smooth space we then argue that the sheaf of relative $2$-differentials is cohomologically flat in dimension $0$; in particular the direct image sheaf will be compatible with the desired base change \cite[132--134]{BS76}. Using the same reasoning, first assuming $S$ smooth, we deduce that the relative Frölicher spectral sequence associated to $f \colon X \to S$ degenerates at $E_1$ \cite[251]{MHS}. Specifically we obtain a canonical injection of sheaves of $\Oka_S$-modules
\begin{equation}\label{pm sheaf map}
f_*\Omega^2_{X/S} \to \Oka_S \otimes_{\csh\ZZ{S}} \Rsh2f{\csh\ZZ{X}}
\end{equation}
whose cokernel is finite locally free.

Let $\Lambda$ be a lattice of rank $r$ and $\mu$ be a $\Lambda$-marking of $\cF$. Then $f_*\Omega^2_{X/S}$ becomes, by virtue of $\mu$, a subsheaf of $\Oka_S$-modules of $\Oka_S \otimes_{\csh\ZZ{S}} \csh\Lambda{S}$ whose cokernel is locally free of rank $r-1$. Thus we obtain---for example, using Grothendieck's theory of flag functors \cite[\S\S2--3]{SHC13-12}---a unique holomorphic map $\tilde h \colon S \to \PP{\Lambda_\CC}$ such that the pullback by $\tilde h$ of the tautological subsheaf of $\Oka_{\PP{\Lambda_\CC}}$-modules
\begin{equation*}
\Oka_{\PP{\Lambda_\CC}}(-1) \subset \Oka_{\PP{\Lambda_\CC}} \otimes_{\csh\ZZ{\PP{\Lambda_\CC}}} \csh\Lambda{\PP{\Lambda_\CC}}
\end{equation*}
yields precisely the image of $f_*\Omega^2_{X/S}$ inside $\Oka_S \otimes_{\csh\ZZ{S}} \csh\Lambda{S}$. We notice that for every point $s \in S$,
\begin{equation}\label{pm pointwise}
\tilde h(s) = (\id\CC \otimes_\ZZ \mu_s)\paren*{\Hdg20{X_s}}
\end{equation}
where $\Hdg20{X_s}$ denotes the canonical image of $\Hsh0{X_s}{\Omega^2_{X_s}}$ in $\CC \otimes_\ZZ \Hsh2{X_s}\ZZ$ and
\[
\id\CC \otimes_\ZZ \mu_s \colon \CC \otimes_\ZZ \Hsh2{X_s}\ZZ \to \CC \otimes_\ZZ \Lambda = \Lambda_\CC
\]
is the complexification of $\mu_s$. \Cref{pm pointwise} implies that the holomorphic map $\tilde h$ factorizes uniquely as $j \circ h$ where $j \colon \pdom\Lambda \inj \PP{\Lambda_\CC}$ denotes the canonical injection and $h \colon S \to \pdom\Lambda$ is a holomorphic map. We call $h$ the \emph{period map} of $(\cF,\mu)$.

\begin{rema}
The literature typically gives \cref{pm pointwise} as the definition of the period map of a marked family. While this is sufficient for families over reduced complex spaces, it is certainly not sufficient as a definition for families over arbitrary complex spaces. Hence in \cref{pm sec} we work with the sheaf of $\Oka_S$-modules $f_*\Omega^2_{X/S}$ rather than with the indexed family of complex vector spaces $\paren{\Hsh0{X_s}{\Omega^2_{X_s}}}_{s\in S}$.
\end{rema}

\begin{rema}\label{pm functorial}
Let $\Lambda$ be a lattice, $\phi = (g,h) \colon \cF' \to \cF$ be a morphism of $\Lambda$-marked families, and $p' \colon S' \to\pdom\Lambda$ and $p \colon S \to \pdom\Lambda$ be the period maps of $\cF'$ and $\cF$, respectively. Then, exploiting the commutativity of \cref{marked mor diagram} and noticing that the sheaf map of \cref{pm sheaf map} is compatible with the base change associated to the square in \cref{mor of families}, we deduce that $p' = p \circ h$. In other words, the period map is functorial.
\end{rema}

\subsection{}\label{deformation sec}
We let $\pt$ denote the complex space whose underlying set is $\set0$ and whose structure sheaf is given by the constant sheaf with value $\CC$.
Let $X$ be a compact complex manifold. Then $X$ can be viewed as a family of compact complex manifolds over $\pt$ by virtue of the unique (constant) holomorphic map $c \colon X \to \pt$. In this spirit a \emph{deformation} of $X$ is a pair $(\cX,\iota)$ where $\cX$ is a family of compact complex manifolds and $\iota \colon (X,\pt,c) \to \cX$ is a morphism of families.

A deformation $(\cX,\iota)$ of $X$ is called \emph{complete} when for every deformation $(\cX',\iota')$ of $X$, where $\iota'$ is a morphism over $j' \colon \pt \to D'$, there exists an open subspace $U \subset D'$ with $j'(0) \in U$ as well as a morphism of families $\phi \colon \cX'_U \to \cX$ such that $\iota = \phi \circ \iota'$.
A deformation $(\cX,\iota)$ of $X$ is called \emph{universal} (resp. \emph{semi-universal}) when it is complete and when for all deformations $(\cX'',\iota'')$ of $X$, where $\iota''$ is a morphism over $j'' \colon \pt \to D''$, and all morphisms of families $\phi_1,\phi_2 \colon \cX'' \to \cX$ over $h_1,h_2 \colon D'' \to D$ that satisfy
\[
\iota = \phi_1 \circ \iota'' \qquad \text{and} \qquad \iota = \phi_2 \circ \iota'',
\]
respectively, there exists an open neighborhood $V$ of $j''(0)$ in $D''$ with ${\rest{h_1}V = \rest{h_2}V}$ (resp. the Jacobian maps of $h_1$ and $h_2$ at the point $j''(0)$ coincide).

We say that $X$ is \emph{unobstructed} or has \emph{unobstructed deformations} when there exists a semi-universal deformation $(\cX,\iota)$ of $X$ such that the base space of the family $\cX$ is smooth.

\begin{rema}\label{local Torelli}
By virtue of its \emph{Kuranishi family} every compact complex manifold $X$ possesses a semi-universal deformation \cite[Theorem~2]{Kur62}.
Assume that $X$ is an ihs manifold.
Then according to Beauville and Bogomolov \cite[771--772]{Bea83} there exists a semi-universal deformation $(\cX,(i,j))$ of $X$ such that $\cX$ is a family of ihs manifolds over a simply connected complex manifold $S$. Moreover for every $\Lambda$-marking $\mu$ of $\cX$, the period map $S \to \pdom\Lambda$ of $(\cX,\mu)$ is a local biholomorphism at the point $j(0)$. The latter fact is usually called the \emph{local Torelli theorem for ihs manifolds} \cite[\nopp 1.15]{Huy99}.
\end{rema}

\subsection{}
Let $\cF$ be a family of compact complex manifolds over $S$ and $s \in S$ be a point. Then the canonical injection of the fiber defines a morphism of families $\iota \colon (\cF(s),\pt,c) \to \cF$ over the map $\pt \to S$ sending $0$ to $s$.
In that regard we say that the family $\cF$ is \emph{complete} (resp. \emph{semi-universal}, resp. \emph{universal}) at $s$ when $(\cF,\iota)$ is a complete (resp. semi-universal, resp. universal) deformation of $\cF(s)$.

\section{Families of ihs manifolds over smooth rational curves}\label{systematic sec}

\subsection{Overview}
In this section we focus our attention on families of ihs manifolds $\cF$ over smooth rational curves. First of all, we note that any such family possesses a marking $\mu$ as defined in \cref{ihs family sec} and admits an invariant, its degree, which turns out to be an integer $d\geq0$. In \cref{degree prop} we show that $d$ can be characterized completely in terms of the period map of $(\cF,\mu)$.

Second of all, we discuss examples of families of low degree. In \cref{isotrivial} we show that families of degree $0$ are trivial. Due to a restriction related to the geometry of the period domain, explained in \cref{no linear subspace}, families of degree $1$ do not exist. In \cref{closed embedding} we show that the period map of a marked family of degree $2$ is an embedding. Thus families of ihs manifolds of degree $2$---in particular, twistor families---are examples of families satisfying the assumptions of \cref{unobs thm}.

Last but not least, with \cref{deg constant,deg constant2} we study the behavior of our notions of degree under deformation. This becomes relevant in \cref{unobs sec}.

\subsection{}\label{line bundle deg}
A \emph{smooth rational curve} is a complex space biholomorphic to $\CP1$. Given a smooth rational curve $C$ and a locally free sheaf of $\Oka_C$-modules $\sL$ of rank $1$, we let $\deg_C(\sL)$ denote the \emph{degree} of $\sL$ on $C$. Note that the resulting map $\deg_C \colon \Pic C \to \ZZ$ is a group isomorphism. When $d$ is an integer, we write $\Oka_C(d)$ for an arbitrary locally free sheaf of $\Oka_C$-modules of rank $1$ whose degree is equal to $d$. 

\begin{defi}\label{family deg}
Let $\cF = (X,C,f)$ be a family of ihs manifolds over a smooth rational curve $C$. We know---compare \cref{pm sec}---that $f_*\Omega^2_{X/C}$ is a locally free sheaf of $\Oka_C$-modules of rank $1$. Thus it makes sense to define
\[
\deg\cF \coloneqq -\deg_C\paren*{f_*\Omega^2_{X/C}}.
\]
We call $\deg\cF$ the \emph{degree} of $\cF$, and we say that $\cF$ is a family (of ihs manifolds) of degree $d$ when $\deg\cF = d$.
\end{defi}

\begin{defi}\label{map deg}
Let $C$ be a smooth rational curve, $V$ a finite-dimensional complex vector space, and $g\colon C\to\PP V$ a holomorphic map to the projective space of lines. Then the \emph{degree} of $g$ is
\[\deg g\defeq\deg_C\paren*{g^*(\Oka_{\PP V}(1))}.\]

When $U$ is a not necessarily open or closed complex subspace of $\PP V$ and the reference to $\PP V$ is understood, we can view a given holomorphic map $h \colon C \to U$ as a holomorphic map $\tilde h \colon C \to \PP V$ by virtue of the canonical injection $j \colon U \inj \PP V$. In that spirit the degree of $h$ is $\deg h \defeq\deg \tilde h$. When, in addition, $C$ is a subspace of $U$, we apply this definition to the canonical injection $i \colon C \inj U$ and speak of a smooth rational curve of degree $\deg i$ in $U$.
\end{defi}

\begin{prop}\label{degree prop}
Let $(\cF,\mu)$ be a $\Lambda$-marked family of ihs manifolds over a smooth rational curve, $h$ be the associated period map. Then $\deg\cF = \deg h$.
\end{prop}
\begin{proof}
Writing $\cF = (X,C,f)$ and letting $\tilde h$ denote the composition of $h$ and the canonical injection $\pdom\Lambda \to \PP{\Lambda_\CC}$, we know that $f_*\Omega^2_{X/C} \cong \tilde h^*\paren*{\Oka_{\PP{\Lambda_\CC}}(-1)}$ by the definition of the period map in \cref{pm sec}. Thus
\[
\deg\cF = -\deg_C\paren*{f_*\Omega^2_{X/C}} = \deg_C\paren*{{\tilde h}^*\paren*{\Oka_{\PP{\Lambda_\CC}}(1)}} = \deg{\tilde h} = \deg h. \qedhere
\]
\end{proof}

\begin{rema}\label{marking exists}
Let $\cF = (X,S,f)$ be a family of ihs manifolds over a nonempty, simply connected space $S$. Then there exist a lattice $\Lambda$ and a $\Lambda$-marking $\mu$ of $\cF$.

Indeed, pick a point $t \in S$. Then there exist a lattice $\Lambda$ and a lattice isomorphism $\nu \colon \Hsh2{X_t}\ZZ \to \Lambda$; for example, take $\Lambda$ equal to $\Hsh2{X_t}\ZZ$ and $\nu = \id\Lambda$. Since the holomorphic map $f\colon X \to S$ is a proper submersion, we know that the sheaf of abelian groups $\Rsh2f{\csh\ZZ X}$ is locally constant. Given that $S$ is simply connected, we infer that the latter sheaf is constant. Therefore we obtain a unique isomorphism of sheaves of abelian groups
\[
\mu \colon \Rsh2f{\csh\ZZ X} \to \csh\Lambda S
\]
for which $\mu_t = \nu$. Since the bilinear forms of the lattices $\Hsh2{X_s}\ZZ$ vary locally constantly\footnote{At this point it is important that in \cref{ihs sec} we have fixed a suitable convention on how to rescale the Beauville--Bogomolov form.} with $s$ in $S$, the map $\mu_s \colon \Hsh2{X_s}\ZZ \to \Lambda$ is a $\Lambda$-marking of $X_s$ for all points $s \in S$. Hence $\mu$ is a $\Lambda$-marking of $\cF$.
\end{rema}

\begin{coro}\label{deg nonneg}
When $\cF$ is a family of ihs manifolds over a smooth rational curve, then $\deg \cF \ge 0$.
\end{coro}
\begin{proof}
Use \cref{marking exists}, \cref{degree prop}, and the fact that the degree of a holomorphic map in the sense of \cref{map deg} is always nonnegative.
\end{proof}

\begin{prop}\label{isotrivial}
Let $\cF = (X,C,f)$ be a family of ihs manifolds over a smooth rational curve. Then $\deg\cF = 0$ if and only if the family $\cF$ is trivial.
\end{prop}  
\begin{proof}
First of all, if $\cF$ is trivial, then $f_*\Omega^2_{X/C} \cong \Oka_C$ and whence $\cF$ is of degree $0$. Conversely now, assume that $\deg\cF=0$. By \cref{degree prop} we know that for every $\Lambda$-marking $\mu$ of $\cF$, the period map of $(\cF,\mu)$ is of degree $0$, whence constant. By virtue of \cref{marking exists} every $\Lambda$-marking of a particular fiber of $\cF$ extends uniquely to a $\Lambda$-marking of $\cF$. Thus for every open subspace $U\subset C$ and every $\Lambda$-marking $\nu$ of $\cF_U$ we see that the period map of $(\cF_U,\nu)$ is locally constant.

Fix a point $s \in C$. According to \cref{local Torelli} there exist a semi-universal deformation $(\cX,\iota)$ of $X_s$ as well as a $\Lambda$-marking $\tilde\nu$ of the family of ihs manifolds $\cX$ such that the period map of $(\cX,\tilde\nu)$ is an open embedding $p \colon S \to \pdom\Lambda$. Since the deformation $(\cX,\iota)$ is complete, there is a connected open neighborhood $U$ of $s$ in $C$ and a morphism of families $\phi\colon\cF_U\to\cX$ over a holomorphic map $h\colon U\to S$. We let $\nu$ denote the unique $\Lambda$-marking of $\cF_U$ for which $\phi$ becomes a morphism of $\Lambda$-marked families; see \cref{induced marking}. Then by \cref{pm functorial} the composition $p \circ h$ is the period map of $(\cF_U,\nu)$, which we know to be constant. We conclude that the holomorphic map $h$ is constant, too, so that the family $\cF_U$ is trivial by virtue of $\phi$.

As $s \in C$ was arbitrary, we have shown that the family $\cF$ is locally trivial. Since the space $C$ is connected, this implies that the family $\cF$ is isotrivial in the sense that every two fibers of $\cF$ are isomorphic. Therefore there exists an ihs manifold $Y$---for example, take an arbitrary fiber of $\cF$---together with an indexed open cover $\mathfrak{U} = (U_i)_{i\in I}$ of $C$ and an indexed family $(\zeta_i)_{i\in I}$ of $U_i$-isomorphisms of families $\zeta_i \colon \cF_{U_i} \to \cY_{U_i}$ where $\cY$ denotes the trivial family of compact complex manifolds $(Y\times C,C,\pr2)$. For all $i,j \in I$ define $U_{ij} \coloneqq U_i\cap U_j$ and
\[
\psi_{ij} \coloneqq (\zeta_i)_{U_{ij}} \circ (\zeta_j^{-1})_{U_{ij}} \colon \cY_{U_{ij}} \to \cY_{U_{ij}}.
\]
Notice that the group $\Hsh0Y{\Theta_Y}$ is trivial for $Y$ is an ihs manifold. Thus the Lie group $A \coloneqq \Aut(Y)$ of holomorphic automorphisms of $Y$ is discrete and we may regard $\psi_{ij}$ as a locally constant map $U_{ij} \to A$. As such $(\psi_{ij})_{i,j \in I}$ is a Čech $1$-cocycle of the constant sheaf of groups $\csh AC$ on $\mathfrak U$.

Since the space $C$ is simply connected and locally pathwise connected, the first Čech cohomology of $\csh AC$ on $\mathfrak U$ is trivial \cite[7.5, 7.13, and 7.14]{Wed16}. Hence there exists an indexed family $(\omega_i)_{i\in I}$ of sections $\omega_i\in\csh AC(U_i)$ so that $\omega_i\psi_{ij} = \omega_j$ on $U_{ij}$ for all $i,j \in I$. Interpreting $\omega_k$ as a $U_k$-automorphism of the family $\cY_{U_k}$, we infer that
\[
\paren*{\omega_i \circ \zeta_i}_{U_{ij}} = \paren*{\omega_j \circ \zeta_j}_{U_{ij}}
\]
for all $i,j \in I$. As a result there exists a $C$-isomorphism of families $\cF \to \cY$.
\end{proof}

\begin{lemm}\label{no linear subspace}
Let $\Lambda$ be a lattice of rank $r\ge3$ and signature $(3,r-3)$. Then the period domain $\pdom\Lambda$ contains no projective subspace of $\PP{\Lambda_{\CC}}$ of positive dimension.
\end{lemm}
\begin{proof}
It is clearly enough to show that $\pdom\Lambda$ contains no projective line. We assume, to the contrary, that there exists a $2$-dimensional complex linear subspace $V \subset \Lambda_{\CC}$ such that $\PP V\subset\pdom\Lambda$. By the definition of the period domain, $x^2 = 0$ and $x{\conj x} > 0$ for all $x \in V \setminus \set0$. The second condition implies that there exists an orthogonal ordered basis $(v,w)$ of $V$ with respect to the Hermitian product $H(x,y)\coloneqq x\conj y$, which is defined on $\Lambda_{\CC}$. The first condition then implies that the quadruple $(v,w,\conj v,\conj w)$ is orthogonal with respect to $H$. Moreover, the entries of this quadruple are strictly positive for $H$, which, however, contradicts the fact that the positive index of inertia of $H$ on $\Lambda_\CC$ is $3$.
\end{proof}

\begin{rema}
Gordon Heier \cite{Heier} has obtained \cref{no linear subspace} for the K3 lattice $\Lambda$ by means of a different argument.
\end{rema}

\begin{coro}\label{deg ne1}
Let $\cF$ be a family of ihs manifolds over a smooth rational curve, then $\deg\cF \ne 1$.
\end{coro}
\begin{proof}
By \cref{marking exists} there exists a marking $\mu$ of $\cF$. Let $h$ denote the associated period map. Then by \cref{degree prop}, $\deg\cF=1$ if and only if $\deg h=1$. The latter condition would imply that $h(C)$ is a projective line in $\PP{\Lambda_\CC}$, which is impossible by \cref{no linear subspace}.
\end{proof}

\begin{prop}\label{upstairs-downstairs}
Let $\cF=(X,C,f)$ be a family of ihs manifolds over a smooth rational curve and $d\in\ZZ$. Then $\deg\cF=d$ if and only if there exists a global section $\sigma$ in the sheaf
\[\Omega^2_{X/C}(d) \coloneqq \Omega^2_{X/C} \otimes_{\Oka_X} f^*\Oka_C(d)\]
that defines a holomorphic symplectic structure on $X_t$ for every $t \in C$.
\end{prop}
\begin{proof}
By \cref{family deg,line bundle deg}, $\deg\cF=d$ if and only if the sheaf of $\Oka_C$-modules $E \coloneqq f_*\Omega^2_{X/C} \otimes_{\Oka_C} \Oka_{C}(d)$ is isomorphic to $\Oka_C$, which is the case if and only if there exists a global nowhere vanishing section in $E$.

By the projection formula the canonical morphism of sheaves of $\Oka_C$-modules $\pi \colon E \to f_*\paren{\Omega^2_{X/C}(d)}$ is an isomorphism. Let $\tau$ be an arbitrary global section in $E$ now, put $\sigma \coloneqq \pi_C(\tau)$, and fix a point $t \in C$. Then $\sigma$ is a global section of $\Omega^2_{X/C}(d)$ which defines a global section $\sigma_t$ of $\Omega^2_{X_t}$. By Grauert's base change theorem we know that $\sigma_t\ne0$ in $\Hsh0{X_t}{\Omega^2_{X_t}}$ if and only if $\tau(t)\ne0$ in $E(t)$. Futhermore, since $X_t$ is an ihs manifold, $\sigma_t\ne0$ if and only if $\sigma_t$ is a holomorphic symplectic structure on $X_t$. Hence the desired equivalence follows.
\end{proof}

\begin{exam}[Twistor families]\label{twistor example}
Let $(M,g,I,J,K)$ be a hyperkähler manifold. Then the twistor construction \cite[554--557]{HKLR} produces a complex manifold $Z$, the \emph{twistor space}, together with a differentiably trivial holomorphic submersion $p \colon Z\to \CP1$. We know there exists a global section $\sigma$ in the sheaf $\Omega^2_{Z/\CP1}(2)$ so that $\sigma$ defines a holomorphic symplectic structure $\sigma_t$ on $Z_t = p^{-1}(t)$ for every point $t \in \CP1$. Therefore, according to \cref{upstairs-downstairs}, when $(M,I)$ is an ihs manifold---or rather, is the almost complex manifold associated to an ihs manifold---the triple $(Z,\CP1,p)$ is a family of ihs manifolds of degree $2$. We call the latter a \emph{twistor family}.
\end{exam}

\begin{prop}\label{closed embedding}
Let $(\cF,\mu)$ be a $\Lambda$-marked family of ihs manifolds over a smooth rational curve $C$ such that $\deg\cF = 2$. Then the period map $h\colon C\to\pdom\Lambda$ of $(\cF,\mu)$ is a closed embedding.
\end{prop}
\begin{proof}
We let $\tilde h$ denote the composition of $h$ and the canonical injection $\pdom{\Lambda}\to\PP{\Lambda_{\CC}}$. By \cref{degree prop}, $\tilde h \colon C \to \PP{\Lambda_\CC}$ is a holomorphic map of degree $2$. Thus $\tilde h$ is given by a linear series $\abs V$ for a nonzero complex linear subspace $V \subset \Hsh0C{\Oka_C(2)}$. The vector space $V$ cannot be of dimension $1$, for if it were, the map $\tilde h$ would be constant and whence of degree $0$. If $V$ were of dimension $2$, the set-theoretic image of $\tilde h$ would be a $1$-dimensional linear subspace of $\PP{\Lambda_\CC}$, which is impossible by \cref{no linear subspace}. Therefore $\tilde h$ is given by the complete linear series $\abs{\Oka_C(2)}$, which implies that $\tilde h$ and whence $h$ are closed embeddings.
\end{proof}

\begin{rema}
If $\cF$ is a family of ihs manifolds of degree $d'$ over $\CP1$ and $g\colon\CP1\to\CP1$ is a branched covering of degree $d$, then $g^*(\cF)$ is a family of degree $d'd$. Taking for $\cF$ a twistor family as in \cref{twistor example}, this shows that there are families of ihs manifolds over $\CP1$ of every even degree $2d>0$. We do not know whether families of odd degree occur.
\end{rema}

\begin{lemm}
\label{deg constant}
Let $(W,S,p)$ be a family of smooth rational curves---that is, a family of compact complex manifolds whose every fiber is a smooth rational curve. Let $\sL$ be a locally free sheaf of $\Oka_W$-modules of rank $1$ and write $i_s \colon W_s \to W$ for the canonical injection when $s \in S$. Then the function $d\colon S\to \ZZ$ given by $d(s) = \deg_{W_s}(i_s^*(\sL))$ is locally constant on $S$.
\end{lemm}
\begin{proof}
By the Riemann--Roch theorem we know that $d(s) = \chi(W_s;i_s^*(\sL)) - 1$ for all $s \in S$. Thus our claim follows from the well-known invariance of the Euler--Poincaré characteristic \cites{EinTheorem,Rie70}.
\end{proof}

\begin{coro}
\label{deg constant2}
We proceed with the notation of \cref{deg constant}.
\begin{enumtheo}
\item\label{deform base} When $V$ is a finite dimensional complex vector space and $g\from W\to\PP V$ is a holomorphic map, then the degree of the maps $g \circ i_s$ is locally constant in $s\in S$.
\item\label{deform family} Let $\cF=(X,W,f)$ be a family of ihs manifolds. Then the degree of the induced families $\cF_{W_s}$ is locally constant in $s\in S$.
\end{enumtheo}
\end{coro}
\begin{proof}
\Cref{deform base} follows from \cref{deg constant} when we apply it to $\sL = g^*(\Oka_{\PP V}(1))$. As a matter of fact, for all $s \in S$,
\[
\deg(g \circ i_s) = \deg_{W_s}\paren*{(g\circ i_s)^*(\Oka_{\PP V}(1))} = \deg_{W_s}\paren*{i_s^*(\sL)}.
\]

Concerning \cref{deform family}, let us write the family $\cF_{W_s}$ and the canonical morphism of families $\cF_{W_s} \to \cF$ as $(X_s,W_s,f_s)$ and $(q,i_s)$, respectively. Then ${q^*(\Omega^2_{X/W}) \cong \Omega^2_{X_s/W_s}}$ since relative differentials are compatible with base change. Moreover, even though $W$ might be nonreduced, we know by \cref{pm sec} that
\[i_s^*\paren*{f_*(\Omega^2_{X/W})} \cong (f_s)_*\paren*{q^*(\Omega^2_{X/W})}.\] 
Hence \cref{deform family} follows when we apply \cref{deg constant} to $\sL = f_*\Omega^2_{X/W}$.
\end{proof}

\section{Proof of the unobstructedness theorem}
\label{unobs sec}

\subsection{Douady space}\label{douady section}
The proof of \cref{unobs thm} makes use of the \emph{Douady space}, a complex analytic analog of the algebraic Hilbert scheme, introduced by Douady \cite{Douady}. When $X$ is a complex space, the Douady space of $X$, denoted  $\dou(X)$, parametrizes the compact complex subspaces of $X$. If $Y \subset X$ is a compact complex subspace, we let $[Y]$ denote the corresponding point in $\dou(X)$.

Recall that there is a closed complex subspace $Z\subset\dou(X)\times X$ universal with the property that the holomorphic map $\rest{\pr1}{Z} \colon Z \to \dou(X)$ is flat and proper. In other words, for every complex space $S$ and every closed complex subspace $Y\subset S\times X$ that is flat and proper over $S$ there exists a unique holomorphic map $b\colon S\to\dou(X)$ such that $Y$ is the pullback of the complex subspace $Z$ under $b\times\id X$.

\begin{theo}\label{q}
Let $r\ge3$ and $d\ge2$ be integers, $V$ be a complex vector space of dimension $r$, and $Q\subset\PP V$ be a smooth quadric hypersurface. Then the set of smooth rational curves of degree $d$ in $Q$ defines a smooth open subspace $S_d(Q) \subset \dou(Q)$ which is either empty or pure of dimension $(r-2)(d+1)-3$.
\end{theo}
\begin{proof}
We divide the proof into the following two steps.
\begin{enumstep}
\item\label{openness} The set $S_d(Q)$ is open in $\dou(Q)$.
\item\label{nbdl} For every smooth rational curve $C\subset Q$ of degree $d$, \[\hdim0C{\Nb C{Q}} = (r-2)(d+1)-3 \quad \text{and} \quad \hdim1C{\Nb C{Q}}=0.\]
\end{enumstep}
\Cref{q} then follows from Kodaira's well-known criterion \cite[Theorem~2]{Kod63}. Without loss of generality we may and do assume that $\PP V = \CP{r-1}$.

\subsubsection*{\Cref{openness}}
Let $s\in \dou(Q)$ be a point corresponding to a smooth rational curve in $Q$. Let $p \colon Z \to \dou(Q)$ be the projection from the universal subspace $Z\subset\dou(Q)\times Q$. Then $p\inv(s)\cong\CP1$ and thus, as $\CP1$ is rigid, there is an open neighborhood $U$ of $s$ in $\dou(Q)$ with $p\inv(t)\cong\CP1$ for all $t\in U$. As $s$ was arbitrary, the set of smooth rational curves in $Q$ defines an open subspace $S\subset\dou(Q)$ and, moreover, the triple $(p\inv(S),S,\rest p{p\inv(S)})$ is a family of smooth rational curves.

Composing the canonical injections $p\inv(S)\inj S\times Q\inj S\times\CP{r-1}$ with the projection onto $\CP{r-1}$, we obtain a holomorphic map $g\colon p\inv(S)\to\CP{r-1}$. Applying \cref{deform base} of \cref{deg constant2} we see that the degree of the image curves $p\inv(t)\inj \CP{r-1}$ is locally constant in $t\in S$. Hence $S_d(Q)$ is an open subset of $\dou(Q)$.

\subsubsection*{\Cref{nbdl}}
Let $C \subset Q$ be a smooth rational curve of degree $d\geq2$. Define $P \cong \CP n$ to be the projective linear subspace of $\CP{r-1}$ spanned by $C$. Then we have an exact sequence of sheaves of $\Oka_C$-modules
\begin{equation}\label[sequ]{nbdlseq}
0\to\Nb C{P}\to\Nb C{\CP{r-1}}\to \rest{\Nb{P}{\CP{r-1}}}C \to0.
\end{equation}
We see that
\[\rest{\Nb {P}{\CP{r-1}}}C \cong \rest{\Oka_{P}(1)^{\oplus(r-1-n)}}C \cong \Oka_{C}(d)^{\oplus(r-1-n)},\]
and \cite[Corollary~1.45]{HM98} implies that
\[\hdim0C{\Nb C{P}}=(n+1)d+n-3 \quad \text{and} \quad \hdim1C{\Nb C{P}}=0.\]
Therefore we deduce that $\hdim0C{\Nb C{\CP {r-1}}}=r(d+1)-4$ and $\hdim1C{\Nb C{\CP {r-1}}}=0$ from the long exact sequence in cohomology associated to \cref{nbdlseq}.

Now we use the exact sequence of sheaves
\begin{equation}\label[sequ]{nbdlseq2}
0\to\Nb C Q\to\Nb C{\CP{r-1}}\to \rest{\Nb Q{\CP{r-1}}}C \to0.
\end{equation}
We see that
\[
\rest{\Nb Q{\CP{r-1}}}C \cong \rest{\Oka_{\CP{r-1}}(2)}C \cong \Oka_C(2d).
\]
Note that $\Nb C Q$ is ample by \cite[Theorem~1]{Bal1} as $d\geq2$. So $\hdim1C{\Nb C Q}=0$. Hence the long exact sequence associated to \cref{nbdlseq2} yields the result.
\end{proof}

\begin{coro}
\label{sd}
Let $d\geq2$ be an integer, $\Lambda$ be a lattice of rank $r\ge 3$ and signature $(3,r-3)$, and $U \subset \pdom\Lambda$ be an open subspace. Then the set of smooth rational curves of degree $d$ in $U$ defines a smooth open subspace $S_d(U) \subset \dou(U)$ which is either empty or pure of dimension $(r-2)(d+1)-3$.
\end{coro}
\begin{proof}
Since the quadratic form of the lattice $\Lambda$ is nondegenerate, we can apply \cref{q} to the hypersurface $Q\subset\PP{\Lambda_{\CC}}$ defined by it. As $U \subset Q$ is open, the canonical map $\dou(U) \to \dou(Q)$ is an open embedding. Regarding the embedding as an inclusion, $S_d(U) = S_d(Q) \cap \dou(U)$. Hence our claim follows.
\end{proof}

\subsection{Overview over the proof of \cref{unobs thm}}
Let ${\cF\defeq (X,C,f)}$ be a family of ihs manifolds with $C\cong\CP1$ and $\mu$ be a $\Lambda$-marking of $\cF$ such that the period map $h \colon C \to \pdom\Lambda$ of $(\cF,\mu)$ is an embedding. Set $d \defeq -\deg_C\paren{f_*{\Omega^2_{X/C}}}$ and $r\defeq\rk\Lambda$. Then by \cref{degree prop} the degree of the map $h$ is equal to $d$ and $d\ge2$ by \cref{deg nonneg,isotrivial,deg ne1}.
Moreover $r$ is the second Betti number of every fiber of $f$. We divide the proof of \cref{unobs thm} into three steps.
\begin{enumstep}
\item\label{ThmApart1} Using the extension theorem, we construct a deformation $(\cX,(i,j))$ of $X$ such that the base space of $\cX$ is smooth and pure of dimension $m\defeq(r-2)(d+1)-3$.
\item\label{ThmApart2} We show that $(\cX,(i,j))$ is complete.
\item\label{ThmApart3} We show that $(\cX,(i,j))$ is semi-universal.
\end{enumstep}

To carry out \cref{ThmApart3} we exploit that $\hdim1X{\Theta_X}=m$. This, as well as the fact that $\hdim0X{\Theta_X}=0$, follow from \cref{tg cohom thm}, which we establish in \cref{tg cohom sec}. Besides, \cref{ThmApart1,ThmApart2,ThmApart3} imply that $X$ has a semi-universal deformation over a smooth complex space; that is, $X$ has unobstructed deformations.

\subsection{\Cref{ThmApart1} in the proof of \cref{unobs thm}}\label{pf step 1}
As the period map $h \colon C \to \pdom\Lambda$ is an embedding, by \cref{ihs extension thm} there exists a family of ihs manifolds $\tilde{\cF} = (\tilde{X},U,\tilde{f})$ together with a $\Lambda$-marking $\tilde{\mu}$ and a holomorphic map $g\colon X \to \tilde{X}$ such that
\begin{enumerate}
\item $U \subset\pdom\Lambda$ is an open subspace containing $h(C)$,
\item the period map of $(\tilde{\cF},\tilde{\mu})$ is the canonical injection $U \inj \pdom\Lambda$,
\item\label{extended family sq} $(g,h) \colon (\cF,\mu) \to (\tilde{\cF},\tilde{\mu})$ is a morphism of $\Lambda$-marked families.
\end{enumerate}
We let $S$ be the set of points in $\dou(U)$ corresponding to smooth rational curves in $U$ of embedding degree $d$. By \cref{sd} we can consider $S$ as a smooth open subspace of $\dou(U)$, which is pure of dimension $m$ for $L_0 \coloneqq h(C) \subset U$ defines a point $[L_0]\in S$. Let $Y\subset\dou(U)\times U$ be the universal subspace, see \cref{douady section}, and define $W\defeq Y\cap(S\times U)$ with projections $\bar{p}\colon W\to S$ and $q\colon W\to U$. Since
\[
W = \setb{([L],l)}{L\subset U \text{ smooth rational curve of degree $d$, }l\in L},
\]
we get a holomorphic map
\begin{align*}
\bar{h}\colon C&\to W\\
x&\mapsto ([L_0],h(x)).
\end{align*}
Evidently, $\bar{h}$ embeds $C$ into $W$ as the fiber of $\bar p$ over $[L_0]$. By the defining properties of $Y$, the map $\bar{p}$ is a proper submersion so that $(\cC,(\bar{h},j))$, with $\cC\defeq(W,S,\bar{p})$, is a deformation of $C$ where $j(0)\defeq[L_0]$.

We consider the pullback family $q^*(\tilde{\cF}) \eqqcolon (Z,W,F)$ together with the canonical morphism of families $(\tilde q,q) \colon q^*(\tilde\cF) \to \tilde\cF$. By the Cartesian property of the latter morphism, as $q\circ\bar{h}=h$, there exists a unique holomorphic map $i \colon X \to Z$ for which $(i,\bar{h})\colon\cF\to q^*(\tilde{\cF})$ is a morphism of families and $(\tilde q,q) \circ (i,\bar h) = (g,h)$. We obtain the following commutative diagram:
\begin{equation*}
\begin{tikzcd}[row sep=normal, column sep=large]
X\dar{f}\rar{i}\ar[bend left=25]{rr}{g}&Z\ar[near start]{r}{\tilde q}\ar[near start]{d}{F}&\tilde{X}\dar{\tilde{f}}\\
C\dar\ar[near end]{r}{\bar{h}}\ar[bend left=25, near start]{rr}{h}&W\rar{q}\dar{\bar{p}}&U\\
\pt\rar{j}&S
\end{tikzcd}
\end{equation*}
Setting $p \coloneqq \bar p\circ F$ and $\cX\defeq(Z,S,p)$, we obtain a deformation $(\cX,(i,j))$ of $X$.

\begin{lemm}\label{submersion}
Let $F\colon Z\to W$ and $\bar{p}\colon W\to S$ be holomorphic maps of complex spaces such that $\bar{p}$ and $\bar{p}\circ F$ are submersions. Let $z\in Z$ be a point and let $Z_s$ and $W_s$ denote the fibers of $\bar{p}\circ F$ and $\bar{p}$ over $s \coloneqq \bar{p}(F(z))$, respectively. Then the following are equivalent:
\begin{enumerate}
\item $F$ is a submersion at $z$.\label{condition1}
\item The induced map $F_s\colon Z_s\to W_s$ is a submersion at $z$.\label{condition2}
\end{enumerate}
\end{lemm}
\begin{proof}
\Cref{condition1} implies \cref{condition2} because base changes preserve submersions. Assume \cref{condition2} now. Moreover, first, assume that $S$ is smooth. Then the spaces $Z$ and $W$ are smooth, too, and for \cref{condition1} it suffices to check that the Jacobian map $\Tsp Fz \colon \Tsp Zz \to \Tsp W{F(z)}$ is surjective. The latter follows from an elementary four-lemma type argument.

Let $S$ be arbitrary now. Without loss of generality we assume that $Z = S\times B'$ and $W = S\times B$ with $B'$ and $B$ being open in $\CC^{n'}$ and $\CC^n$ and $F$ and $\bar p$ being the first projection maps, respectively. Furthermore we may assume that $S$ is a closed subspace of an open subspace $\tilde S \subset \CC^m$. By \cite[0.22, Corollary~2]{Fischer} we see that the holomorphic map $F \colon S\times B' \to S\times B$ is---at least in a neighborhood of the point $z$---induced by a holomorphic map $\tilde F \colon \tilde S\times B' \to \tilde S\times B$. The map $\tilde F$ can be chosen so that it commutes with the projections to $\tilde S$. The smooth case thus implies that $\tilde F$ is a submersion at $z$. Hence we obtain \cref{condition1}.
\end{proof}

In \cref{pf step 2} below we need that every deformation of $X$ lifts to a deformation of the holomorphic map $f\colon X \to C$ in the sense of Ran \cite[Definition~1.1 and §3]{Ran89}. Hence we recall a theorem on the stability of fiber structures.

\begin{theo}\label{horikawa}
Let $(X,Y,f)$ be a family of compact complex manifolds over a compact complex manifold $Y$ such that $f_*\Oka_X \cong \Oka_Y$ and $\Rsh1f{\Oka_X}=0$. Let $(\cX,(i,j))$ be a deformation of $X$ with $\cX=(Z,S,p)$. Then, after possibly shrinking $S$ around $j(0)$, there are a deformation $(\cY,(\bar i,j))$ of $Y$ with $\cY=(W,S,\bar{p})$ and a holomorphic map $F\colon Z\to W$ such that $(i,\bar i) \colon (X,Y,f) \to (Z,W,F)$ is a morphism of families of compact complex manifolds and $\bar p \circ F = p$.
\end{theo}
\begin{proof}
When $S$ is smooth, this is due to Kodaira \cite[87]{Kod63}. For arbitrary $S$ the methods of Ran \cite[Theorem~2.1]{Ran91} imply the existence of $\cY$, $\bar i$, and $F$ subject to all stipulated properties except for $F$ being a submersion. Since however $f \colon X\to Y$ is a submersion, \cref{submersion} implies that $F$ is a submersion at all points of $i(X) \subset Z$. Thus exploiting the properness of $p \colon Z\to S$, we can shrink $S$ further in order to make $F\colon Z\to W$ a submersion entirely.
\end{proof}

\subsection{\Cref{ThmApart2} in the proof of \cref{unobs thm}}\label{pf step 2}
To prove that the constructed deformation is complete, let there be given another deformation $(\cX',(i',j'))$ of $X$ with $\cX'=(Z',S',p')$. Up to shrinking $S'$ around $j'(0)$, we have to construct a morphism of families $(a,b)\colon\cX'\to\cX$ such that $(i,j) = (a,b) \circ (i',j')$. For the reader's convenience, all spaces and morphisms appearing in this construction are pictured in \cref{big diagram}.

As ihs manifolds are simply connected, $f_*\Oka_{X}\cong\Oka_C$ and $\Rsh1{f}{\Oka_{X}}=0$ so that we can apply \cref{horikawa} to the family $\cF$ and the deformation $(\cX',(i',j'))$ of $X$. After possibly shrinking $S'$ around $j'(0)$, we obtain a deformation $(\cC',(\bar{h}',j'))$ of $C$ with $\cC'=(W',S',\bar{p}')$ as well as a family of compact complex manifolds $\cF' \coloneqq (Z',W',F')$ such that $(i',\bar h') \colon \cF \to \cF'$ is a morphism of families and $p' = \bar p' \circ F'$.
By \cref{torellilike}, since the period map of $(\tilde\cF,\tilde\mu)$ is the canonical injection $U \inj \pdom\Lambda$, the family $\tilde\cF$ is semi-universal at all points of $U$.
Hence we may apply \cref{universal ex} of \cref{universal map} and after another shrinking of $S'$ around $j'(0)$ there exists a morphism of families $(g',h') \colon \cF' \to \tilde{\cF}$ such that $(g,h) = (g',h') \circ (i',\bar{h}')$.

As $C\isom\CP1$ is rigid, we can shrink $S'$ around $j'(0)$ so that $\cC'$ is a family of smooth rational curves. By \cref{deform base} of \cref{deg constant2} we can assume that, for every point $s \in S'$, the composition
\[
W'_s = \bar p'^{-1}(s) \inj W' \overset{h'}\longrightarrow U \inj \PP{\Lambda_\CC}
\]
is an embedding of degree $d$. In particular the holomorphic map
\[
(\bar p',h') \colon W' \to S' \times U
\]
is an embedding. Therefore, by virtue of the universal property of the Douady space, see \cref{douady section}, there exists a unique morphism of families $(b',b)\colon\cC'\to\cC$ such that $b'$ commutes with the projections to $U$; that is, $h' = q \circ b'$. Thus
\[
q\circ\bar{h} = h = h'\circ\bar{h}' = (q\circ b')\circ\bar{h}' = q\circ (b'\circ\bar{h}').
\]
Exploiting the uniqueness part of the universal property of the Douady space, we see that $(\bar h,j) = (b',b) \circ (\bar h',j')$.

Just like in \cref{pf step 1} above, using the Cartesian property of the canonical morphism of families $(\tilde q,q) \colon q^*(\tilde\cF) \to \tilde\cF$, we deduce the existence of a unique holomorphic map $a \colon Z' \to Z$ for which $(a,b')\colon \cF' \to q^*(\tilde\cF)$ a morphism of families such that ${(g',h') = (\tilde q,q) \circ (a,b')}$. As a consequence $(a,b)\colon\cX'\to\cX$ is a morphism of families. We have seen that $j = b \circ j'$. 
Using again the Cartesian property of $(\tilde q,q)\colon q^*(\tilde\cF) \to \tilde\cF$, we deduce that $i=a\circ i'$ for
\[
\tilde q \circ i = g = g' \circ i' = (\tilde q \circ a) \circ i' = \tilde q \circ (a \circ i').
\]
This proves that the deformation $(\cX,(i,j))$ of $X$ is complete.

\begin{figure}[ht]
\centering
\begin{tikzcd}[row sep=large, column sep=huge]
X\rar{i'}\dar{f}\ar[bend left=20]{rr}{i}&Z'\rar{a}\dar{F'}\ar[bend left=20]{rr}{g'}\ar[bend left=50, near end]{dd}{p'}&Z\ar[near start]{r}{\tilde q}\dar{F}\ar[bend left=50, near end]{dd}{p}&\tilde{X}\dar{\tilde{f}}\\
C\rar{\bar{h}'}\ar[bend left=20, near start]{rr}{\bar{h}}\dar&W'\rar{b'}\dar{\bar{p}'}\ar[bend left=20, near end]{rr}{h'}&W\rar{q}\dar{\bar{p}}&U\\
\pt\rar{j'}\ar[bend right=20]{rr}{j}&S'\rar{b}&S
\end{tikzcd}
\caption{The complex spaces and holomorphic maps that occur in \cref{ThmApart2} of the proof of \cref{unobs thm}. The diagram is commutative. In addition, the evident squares are Cartesian.}\label{big diagram}
\end{figure}
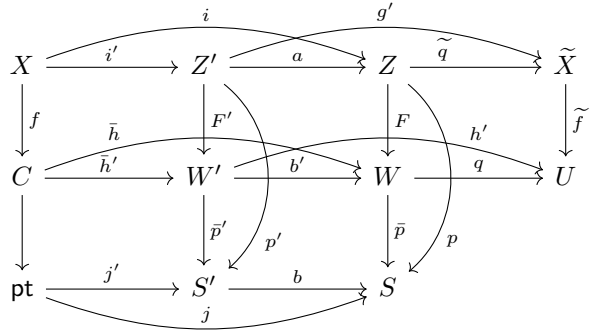

\subsection{\Cref{ThmApart3} in the proof of \cref{unobs thm}}\label{pf step 3}
By \cref{pf step 2} we know that the deformation $(\cX,(i,j))$ of $X$ is complete. Hence, its associated Kodaira--Spencer map
\[
\kappa \colon \Tsp S{j(0)} \to \Hsh1X{\Theta_X}
\]
is surjective \cite[cf.][Remark~5.2]{Cat88}. By \cref{pf step 1}, the complex space $S$ is smooth and of dimension $m=(r-2)(d+1)-3$ at $j(0)$; in particular, $\Tsp S{j(0)}$ is of dimension $m$. Since by \cref{tg cohom thm} the vector space $\Hsh1X{\Theta_X}$ is of dimension $m$, too, we infer that $\kappa$ is a bijection. This proves that $(\cX,(i,j))$ is a semi-universal deformation of $X$ \cite[loc.~cit.]{Cat88}.

\section{Tangent cohomology of families of ihs manifolds}
\label{tg cohom sec}

\subsection{}
Throughout \cref{tg cohom sec}, employing \cref{family deg}, we assume that $\cF \defeq (X,\CP1,f)$ is a family of ihs manifolds of degree $d$. By \cref{deg nonneg}, $d\ge0$.
Furthermore we know that for every integer $k$ the sheaf of abelian groups $\Rsh kf{\csh\ZZ X}$ is locally constant, whence constant, on $\CP1$. Let $b_k$ denote the rank of $\Rsh kf{\csh\ZZ X}$ which is likewise the $k$th Betti number of all fibers of $\cF$.

Our goal is to prove \cref{tg cohom thm}. We use the relative Frölicher spectral sequence, the Leray spectral sequence, and the relative cotangent sequence associated to the holomorphic map $f \colon X\to \CP1$ for that matter.

\begin{prop}
\label{relative froelicher}
\begin{enumtheo}
\item\label{RqOmegap free} Let $(p,q) \in \NN^2$ and $s \in \CP1$. Then $\Rsh qf{\Omega^p_{X/\CP1}}$ is a locally free sheaf of $\Oka_{\CP1}$-modules of rank $\hodge pq{X_s} \coloneqq \hdim q{X_s}{\Omega^p_{X_s}}$.
\item\label{froelicher degenerates} The relative Frölicher spectral sequence of $\cF$ degenerates at $E_1$.
\item\label{relative duality} $\Rsh2f{\Oka_X} \cong \Oka_{\CP1}(d)$ and $\Rsh1f{\Omega^1_{X/\CP1}} \cong (\Rsh1f{\Omega^1_{X/\CP1}})^\vee$ in the sense of sheaves of $\Oka_{\CP1}$-modules.
\item\label{ed}
\[
\hdim i{\CP1}{\Rsh1f{\Omega^1_{X/\CP1}}\otimes\Oka_{\CP1}(d)}= \begin{cases}(b_2-2)(d+1)&\text{when }i=0,\\0&\text{when }i=1.\end{cases}
\]
\item\label{b3} There is a short exact sequence of sheaves of $\Oka_{\CP1}$-modules
\[
0 \to \Rsh1f{\Omega^2_{X/\CP1}} \to \Oka^{\oplus b_3}_{\CP1} \to \Rsh2f{\Omega^1_{X/\CP1}} \to 0 .
\]
\end{enumtheo}
\end{prop}

\begin{proof}
Since every fiber of the family of compact complex manifolds $\cF$ is Kählerian, the function $h \colon \CP1 \to \NN$ given by $h(t) = \hodge pq{X_t}$ is constant \cite[Proposition 9.20]{Voisin}. Thus by Grauert's base change theorem \cite{EinTheorem} the sheaf of $\Oka_{\CP1}$-modules $\Rsh qf{\Omega^p_{X/\CP1}}$ is locally free and the evident base change map
\[
\CC \otimes_{\Oka_{\CP1,s}} \paren{\Rsh qf{\Omega^p_{X/\CP1}}}_s \to \Hsh q{X_s}{\Omega^p_{X_s}}
\]
is an isomorphism of complex vector spaces. This proves \cref{RqOmegap free}. Furthermore, we obtain \cref{froelicher degenerates} as a consequence of these facts \cite[251]{MHS}.

\Cref{relative duality}. The bilinear forms of the lattices $\Hsh2{X_t}\ZZ$, see \cref{ihs sec}, vary locally constantly with $t \in \CP1$, thus yield a symmetric $\ZZ$-bilinear sheaf map
\[
\Rsh2f{\csh\ZZ X} \times \Rsh2f{\csh\ZZ X} \to \csh\ZZ{\CP1}
\]
which is stalkwise nondegenerate. Extending the scalars by virtue of the morphism of sheaves of rings $\csh\ZZ{\CP1} \to \Oka_{\CP1}$, we obtain a nondegenerate symmetric $\Oka_{\CP1}$-bilinear sheaf map
\[
\sG \times \sG \to \Oka_{\CP1}, \quad \text{with} \quad \sG := \Oka_{\CP1} \otimes_{\csh\ZZ{\CP1}} \Rsh2f{\csh\ZZ X}\isom\Oka_{\CP1}^{\oplus b_2}.
\]
Let us write $(\sF^j)_{j \in \NN}$ for the Hodge filtration on $\sG$ \cite[loc.~cit.]{MHS}. Then $\sF^1$ is the perpendicular space of $\sF^2$ with respect to the latter pairing, and vice versa, since the same holds for every Beauville--Bogomolov form \cite[(1) on \pno~773]{Bea83}. Consequently we dispose of isomorphisms of sheaves of $\Oka_{\CP1}$-modules
\begin{gather*}
\sG/\sF^1 \to \sHom_{\Oka_{\CP1}}(\sF^2,\Oka_{\CP1}) = (\sF^2)^\vee,\\
\sF^1/\sF^2 \to \sHom_{\Oka_{\CP1}}(\sF^1/\sF^2,\Oka_{\CP1}) = (\sF^1/\sF^2)^\vee .
\end{gather*}
By \cref{froelicher degenerates} we know that
\begin{align*}
& \sG/\sF^1 \cong \Rsh2f{\Oka_X}, && \sF^1/\sF^2 \cong \Rsh1f{\Omega^1_{X/\CP1}}, && \text{and} & \sF^2 \cong f_*\Omega^2_{X/\CP1}.
\end{align*}
Hence the claim follows from our assumption that $\deg\cF = d$ which entails that $f_*\Omega^2_{X/\CP1}\cong\Oka_{\CP1}(-d)$.

\Cref{ed}. The results of \cref{froelicher degenerates,relative duality} give rise to the two short exact sequences of sheaves of $\Oka_{\CP1}$-modules
\begin{gather}
\label[sequ]{seq1}0\to \sF^1 \to\Oka_{\CP1}^{\oplus b_2} \to \Oka_{\CP1}(d) \to 0,\\
\label[sequ]{seq2}0\to \Oka_{\CP1}(-d) \to \sF^1 \to \Rsh1f{\Omega^1_{X/\CP1}} \to 0.
\end{gather}
By the Birkhoff--Grothendieck theorem the locally free sheaf of $\Oka_{\CP1}$-modules $\sF^1$ is isomorphic to a finite direct sum $\bigoplus\Oka_{\CP1}(a_\nu)$ for some integers $a_\nu\in\ZZ$. According to \cref{seq1}, $\sF^1$ is a subsheaf of a trivial sheaf of $\Oka_{\CP1}$-modules, so $a_\nu\leq0$ for all $\nu$. Moreover, \cref{seq1} shows that $\rk\sF^1=b_2-1$ and that the sum of the $a_\nu$ is $-d$. Thus $0\leq a_\nu+d\leq d$ for all $\nu$, whence
\begin{align*}
\hdim0{\CP1}{\sF^1(d)}&=\sum_{\nu=1}^{b_2-1} (a_\nu+d+1)= -d + (b_2-1)(d+1), \\
\hdim1{\CP1}{\sF^1(d)}&=\hdim0{\CP1}{\bigoplus\Oka_{\CP1}(-(a_\nu+d)-2)}=0.
\end{align*}
Tensoring \cref{seq2} with $\Oka_{\CP1}(d)$ and passing to the associated long exact sequence in cohomology, we see that
\begin{align*}
\hdim0{\CP1}{\Rsh1f{\Omega^1_{X/\CP1}}\otimes\Oka_{\CP1}(d)}&=\hdim0{\CP1}{\sF^1(d)}-1=(b_2-2)(d+1),\\\hdim1{\CP1}{\Rsh1f{\Omega^1_{X/\CP1}\otimes\Oka_{\CP1}(d)}}&=\hdim1{\CP1}{\sF^1(d)}=0,
\end{align*}
which proves the result.

\Cref{b3}. Consider the Hodge filtration $(\sF^j)_{j\in\NN}$ on the sheaf of $\Oka_{\CP1}$-modules
\[
\sH \coloneqq \Oka_{\CP1} \otimes_{\csh\ZZ{\CP1}} \Rsh3f{\csh\ZZ X} \cong \Oka_{\CP1}^{\oplus b_3}
\]
now. Then by \cref{froelicher degenerates}
\begin{align*}
\sH/\sF^1 &\cong \Rsh3f\Oka_X, & \sF^1/\sF^2 &\cong \Rsh2f{\Omega^1_{X/\CP1}},\\
\sF^3 &\cong f_*\Omega^3_{X/\CP1}, & \sF^2/\sF^3 &\cong \Rsh1f{\Omega^2_{X/\CP1}}.
\end{align*}
By \cref{RqOmegap free} the two sheaves on the left-hand side vanish, because the $(0,3)$ and $(3,0)$ Hodge numbers of every ihs manifold vanish \cite[762--764]{Bea83}. Hence we deduce the desired sequence from the short exact sequence 
\[
0\to \sF^2\to \sF^1\to \sF^1/\sF^2\to0. \qedhere
\]
\end{proof}

\begin{prop}\label{Omega-2 relative}
$\Theta_{X/\CP1} \cong \Omega^1_{X/\CP1} \otimes f^*\Oka_{\CP1}(d)$ as sheaves of $\Oka_X$-modules.
\end{prop}

\begin{proof}
From \cref{upstairs-downstairs} we know that there exists a global section $\sigma$ in the sheaf $\Omega^2_{X/\CP1} \otimes f^*\Oka_{\CP1}(d)$ such that for every $t \in \CP1$ the pullback of $\sigma$ defines a holomorphic symplectic structure on the fiber $X_t$. Therefore the contraction with $\sigma$, which is a morphism of sheaves of $\Oka_X$-modules
\[
\Theta_{X/\CP1} \to \Omega^1_{X/\CP1} \otimes f^*\Oka_{\CP1}(d),
\]
is an isomorphism.
\end{proof}

\begin{rema}
\label{canonical}
Assume that the fibers of $\cF$ are of dimension $2n$ with $n \in \NN$. Then taking determinants we can use \cref{Omega-2 relative} and the relative cotangent sequence to calculate the relative canonical sheaf and the canonical sheaf of $X$ over $\CP1$ and $X$, respectively. We find that
\[
\omega_{X/\CP1}\cong f^*\Oka_{\CP1}\left(-nd\right) \quad \text{and} \quad \omega_X\cong f^*\Oka_{\CP1}\left(-2-nd\right).
\]
\end{rema}

\begin{theo}\label{tg cohom thm}
Assume that for a $\Lambda$-marking $\mu$ of $\cF$ the associated period map $h \colon \CP1 \to \pdom\Lambda$ is an immersion. Then $d\ge2$ and
\begin{equation*}
\hdim iX{\Theta_X} = \begin{cases}
0  & \text{when } i=0,\\
(b_2-2)(d+1)-3 & \text{when } i=1.
\end{cases}
\end{equation*}
Moreover if $b_3\neq0$, the obstruction space $\Hsh2X{\Theta_X}$ is nontrivial. If $\cF$ is a family of K3 surfaces---that is, if the fibers of $\cF$ are of dimension $2$---then
\begin{align*}
\hdim iX{\Theta_X} = \begin{cases}
d+3  & \text{when } i=2,\\
0 & \text{when } i=3.
\end{cases}
\end{align*}
\end{theo}
\begin{proof}
We have already noted that $d\ge0$. By \cref{degree prop} we know that $\deg h = d$. Thus since $h$ is not constant, $d\ne0$. By virtue of \cref{deg ne1} we conclude that $d\ge2$.

Now consider the relative tangent sequence
\[
0 \to \Theta_{X/\CP1} \to \Theta_X \to f^*\Theta_{\CP1} \to 0,
\]
which is short exact for $f\colon X\to \CP1$ is a submersion. We contend that the beginning of the associated long exact sequence with respect to the pushforward by $f$ has the following form:
\begin{equation}\label[sequ]{longseq}
\begin{tikzcd}[column sep=small]
0 \rar & 0 \rar & \Oka_{\CP1}(2) \ar[out=-30, in=150]{dll} \\
\Rsh1f{\Omega^1_{X/\CP1}}\otimes\Oka_{\CP1}(d) \rar[] & \Rsh1f{\Theta_X} \rar & 0 \ar[out=-30, in=150]{dll} \\
\Rsh2f{\Omega^1_{X/\CP1}}\otimes\Oka_{\CP1}(d) \rar[] & \Rsh2f{\Theta_X} \rar & \Oka_{\CP1}(d+2)
\end{tikzcd}
\end{equation}

Indeed, for all $i \in \ZZ$, we can calculate the sheaf $\Rshp if{f^*\Theta_{\CP1}}$ by means of the projection formula:
\[
\Rshp if{f^*\Theta_{\CP1}} \cong \Rshp if{\Oka_X \otimes f^*\Oka_{\CP1}(2)} \cong \Rsh if{\Oka_X} \otimes \Oka_{\CP1}(2).
\]
Since the $(0,0)$ and $(0,1)$ Hodge numbers of ihs manifolds are equal to $1$ and $0$, respectively, using \cref{RqOmegap free,relative duality} of \cref{relative froelicher}, we obtain that
\[
\Rsh if{(f^*\Theta_{\CP1})}\cong\begin{cases}
\Oka_{\CP1}(2) & \text{when } i=0,\\
0 & \text{when } i=1,\\
\Oka_{\CP1}(d+2) & \text{when } i=2.\\
\end{cases}
\]
By \cref{Omega-2 relative} and the projection formula,
\[
\Rsh i f{\Theta_{X/\CP1}}\cong\Rsh i f{\Omega^1_{X/\CP1}}\otimes\Oka_{\CP1}(d)
\]
for all $i \in \ZZ$; in particular $f_*\Theta_{X/\CP1}=0$ because $\hodge10{X_s}=0$ for $s \in \CP1$ arbitrary.

According to Griffiths's interpretation \cite[(1.20)]{Gri68}, the Jacobian map at $s$ of our period map $h$ factorizes over the Kodaira--Spencer map $\kappa$ as follows:
\[
\begin{tikzcd}[column sep=small]
\Tsp{\CP1}s \arrow{rr}{\Tsp hs} \arrow{dr}{\kappa} & & \Tsp{\pdom\Lambda}{h(s)} \\
& \Hsh1{X_s}{\Theta_{X_s}} \ar[dotted]{ru}{\exists\gamma}
\end{tikzcd}
\]
Since we assumed $h$ to be an immersion, $\Tsp hs$ and whence $\kappa$ are injective. Thus the connecting homomorphism
\[\Theta_{\CP1} \cong f_*(f^*\Theta_{\CP1}) \to \Rsh1f{\Theta_{X/\CP1}}\]
in our long exact sequence is an injective sheaf map. With $f_*\Theta_{X/\CP1}=0$ we conclude that $f_*\Theta_X=0$. In addition, \cref{longseq} gives rise to a short exact sequence of sheaves of $\Oka_{\CP1}$-modules
\[
0\to\Oka_{\CP1}(2)\to\Rsh1f{\Omega^1_{X/\CP1}}\otimes\Oka_{\CP1}(d)\to\Rsh1f{\Theta_X}\to0.
\]
Passing to the long exact sequence in cohomology and applying \cref{ed} of \cref{relative froelicher}, we deduce that
\begin{align*}
\hdim0{\CP1}{\Rsh1f{\Theta_X}} &= \hdim0{\CP1}{\Rsh1f{\Omega^1_{X/\CP1}}\otimes\Oka_{\CP1}(d)}-\hdim0{\CP1}{\Oka_{\CP1}(2)} \\
&=(b_2-2)(d+1)-3,\\
\hdim1{\CP1}{\Rsh1f{\Theta_X}} &= \hdim1{\CP1}{\Rsh1f{\Omega^1_{X/\CP1}}\otimes\Oka_{\CP1}(d)}=0.
\end{align*}

Next we consider the Leray spectral sequence
\[E_2^{pq} \cong \Hsh p{\CP1}{\Rsh qf\Theta_X} \Rightarrow \Hsh{p+q}X{\Theta_X}\]
for the holomorphic map $f \colon X\to \CP1$ and the sheaf of $\Oka_X$-modules $\Theta_X$.
The entries $E_2^{pq}$ are zero for $p>1$, because $\CP1$ has dimension $1$ and the sheaves of $\Oka_{\CP1}$-modules $\Rsh qf\Theta_X$ are coherent. Hence the spectral sequence degenerates on $E_2$ and we get the following formulae:
\begin{align*}
\hdim0X{\Theta_X} &= \hdim0{\CP1}{f_*\Theta_X}=0, \\
\hdim1X{\Theta_X} &= \hdim0{\CP1}{\Rsh1f{\Theta_X}}+\hdim1{\CP1}{f_*\Theta_X} = (b_2-2)(d+1)-3, \\
\hdim2X{\Theta_X} &= \hdim0{\CP1}{\Rsh2f{\Theta_X}}+\hdim1{\CP1}{\Rsh1f{\Theta_X}} = \hdim0{\CP1}{\Rsh2f{\Theta_X}},\\
\hdim3X{\Theta_X} &= \hdim0{\CP1}{\Rsh3f{\Theta_X}}+\hdim1{\CP1}{\Rsh2f{\Theta_X}}.
\end{align*}

Assume that $b_3\ne 0$ now. Notice that $b_3$ is an even number and that in the short exact sequence of \cref{b3} of \cref{relative froelicher} the ranks of the locally free sheaves of $\Oka_{\CP1}$-modules $\Rsh1f{\Omega^2_{X/\CP1}}$ and $\Rsh2f{\Omega^1_{X/\CP1}}$ are both equal to $b_3/2$, which is a number strictly less than $b_3$. In particular there exists a global section in the sheaf $\Oka_{\CP1}^{\oplus b_3}$ that does not vanish in the quotient $\Rsh2f{\Omega^1_{X/\CP1}}$. Thus $\Rsh2f{\Omega^1_{X/\CP1}}\otimes\Oka_{\CP1}(d)$ has a nontrivial global section, too. Invoking \cref{longseq} we conclude that
\[
\hdim2X{\Theta_X} = \hdim0{\CP1}{\Rsh2f{\Theta_X}} \ge \hdim0{\CP1}{\Rsh2f{\Omega^1_{X/\CP1}}\otimes\Oka_{\CP1}(d)} > 0.
\]

Last but not least, drop the assumption that $b_3\ne0$ and assume that $\cF$ is a family of K3 surfaces instead. Then $\Rsh if{\Omega^1_{X/\CP1}}=0$ for $i \in \set{2,3}$ and $\Rsh3f{\Oka_X} = 0$ given that the corresponding Hodge numbers of K3 surfaces vanish. Hence writing out the next line in the long exact \cref{longseq}, we see that $\Rsh2f{\Theta_X} \cong \Oka_{\CP1}(d+2)$ and $\Rsh3f{\Theta_X} = 0$. Therefore
\begin{align*}
\hdim2X{\Theta_X}&=\hdim0{\CP1}{\Oka_{\CP1}(d+2)}=d+3,\\
\hdim3X{\Theta_X}&=\hdim1{\CP1}{\Oka_{\CP1}(d+2)}=0. \qedhere
\end{align*}
\end{proof}

\section{Universal morphisms along subspaces}
\label{universal map sec}

\subsection{}\label{universality intro}
Consider a compact complex manifold $X$ for which every global holomorphic vector field on $X$ is trivial.
Then every semi-universal deformation of $X$ is likewise universal \cite[\nopp I.10.5--6]{BPV84}. Yet, more is true.

Consider, for $k \in \set{1,2}$, a deformation $(\cX_k,\iota_k)$ of $X$ such that $\cX_k$ is a family over $D_k$ and $\iota_k$ is a morphism of families over the holomorphic map $b_k \colon \pt \to D_k$. Assume that $(\cX_1,\iota_1)$ is semi-universal, whence universal. Moreover let $U$ and $U'$ be open subspaces of $D_2$ containing the point $b_2(0)$ and let
\[
\phi = (g,h) \colon (\cX_2)_U \to \cX_1 \quad \text{and} \quad \phi' = (g',h') \colon (\cX_2)_{U'} \to \cX_1
\]
be morphisms of families satisfying $\iota_1 = \phi \circ \iota_2$ and $\iota_1 = \phi' \circ \iota_2$, respectively. 
Then not only does there exist an open subspace $V \subset U\cap U'$ so that $b_2(0) \in V$ and $\rest{h}V = \rest{h'}V$, but there exists an open subspace $W \subset U\cap U'$ so that $b_2(0) \in W$ and $\rest\phi W = \rest{\phi'}W$ as morphisms of families $(\cX_2)_W \to \cX_1$ \cite[cf.][Remark after I.10.6]{BPV84}.

For later reference we briefly elaborate on the argument.

\begin{theo}\label{rigidity}
Let $\cF$ be a family of compact complex manifolds over $S$ and let ${\psi \colon \cF \to \cF}$ be an $S$-morphism of families so that $\psi(t) = \id{\cF(t)}$ for a point $t \in S$. Assume that $\Hsh0{\cF(t)}{\Theta_{\cF(t)}} = \set0$. Then there exists an open subspace $W \subset S$ such that $t \in W$ and $\psi_W = \id{\cF_W} \colon \cF_W \to \cF_W$.
\end{theo}
\begin{proof}
See Looijenga and Peters \cite[170]{LP80}.
\end{proof}

\begin{coro}\label{eta=eta' coro}
Let $\cF_1$ and $\cF_2$ be two families of compact complex manifolds and $\phi,\phi' \colon \cF_2 \to \cF_1$ be two morphisms of families over the same base map $h \colon D_2 \to D_1$. Assume that $\Hsh0{\cF_2(s)}{\Theta_{\cF_2(s)}} = \set0$ for all $s \in D_2$. Then the set
\[
W \coloneqq \setb*{s \in D_2}{\phi(s) = \phi'(s) \colon \cF_2(s) \to \cF_1(h(s))}
\]
is open in $D_2$. Moreover, the morphisms $\phi$ and $\phi'$ restrict to the same morphism of families $(\cF_2)_W \to \cF_1$.
\end{coro}
\begin{proof}
We know that a pullback $\eta \colon h^*(\cF_1) \to \cF_1$ of the family of compact complex manifolds $\cF_1$ by $h$ exists; see \cref{pullback sec}. Thus $\phi$ and $\phi'$ factorize uniquely as $\phi = \eta\circ \bar\phi$ and $\phi' = \eta\circ \bar\phi'$, respectively, where $\bar\phi$ and $\bar\phi'$ are $D_2$-isomorphisms. Define $\psi \coloneqq \bar\phi^{-1} \circ \bar\phi'$. Then $\psi$ is a $D_2$-automorphism of $\cF_2$ and $W$ is precisely the set of points $s \in D_2$ for which $\psi$ induces the identity on the fiber over $s$. Applying \cref{rigidity}, we see that $W$ is an open subset of $D_2$ and $\psi_W \colon (\cF_2)_W \to (\cF_2)_W$ is the identity. The latter implies that $\bar\phi_W = {\bar\phi'}_W$, which subsequently implies that $\phi$ and $\phi'$ restrict to the same morphism $(\cF_2)_W \to \cF_1$.
\end{proof}

The goal of \cref{universal map sec} is to generalize \cref{universality intro} to situations where not a single compact complex manifold $X$ is given but a family of compact complex manifolds $\cF$ over an arbitrary base space $S$. Precisely we prove the following; when $S = \pt$, we recover the well-known facts of \cref{universality intro}.

\begin{theo}\label{universal map}
Let $\cF$, $\cF_1$, and $\cF_2$ be families of compact complex manifolds over $S$, $D_1$, and $D_2$, respectively, and $\iota_k \colon \cF \to \cF_k$ be a morphism of families over $b_k$ for $k \in \set{1,2}$. Assume that $\Hsh0{\cF(s)}{\Theta_{\cF(s)}}=\set0$ and that $\cF_1$ is semi-universal at $b_1(s)$ for all $s \in S$. Moreover assume that $b_2 \colon S\to D_2$ is an embedding into a second-countable space $D_2$. Then:
\begin{enumtheo}
\item\label{universal ex} There exists an open subspace $W$ of $D_2$ together with a morphism of families $\phi \colon (\cF_2)_W \to \cF_1$ such that $b_2(S) \subset W$ and $\iota_1 = \phi \circ \iota_2$.
\item\label{universal un} When $W'$ is another open subspace of $D_2$ and $\phi' \colon (\cF_2)_{W'} \to \cF_1$ is a morphism of families such that $b_2(S) \subset W'$ and $\iota_1 = \phi'\circ\iota_2$, then $\phi$ and $\phi'$ agree on an open subspace $W''$ of $W\cap W'$ with $b_2(S) \subset W''$.
\end{enumtheo}
\end{theo}

\subsection{}\label{universal map idea}
The idea of the proof of \cref{universal map} is straightforward. Pick a point $s \in S$. Then we dispose of biholomorphisms $\iota_k(s) \colon \cF(s) \to \cF_k(b_k(s))$ for $k \in \set{1,2}$ and hence of a biholomorphism
\begin{equation}\label{fiber map}
\iota_1(s) \circ \iota_2(s)^{-1} \colon \cF_2(b_2(s)) \to \cF_1(b_1(s)).
\end{equation}
Exploiting the completeness of the family $\cF_1$ at $b_1(s)$, we obtain a morphism of families $\psi \colon (\cF_2)_U \to \cF_1$ defined on an open neighborhood $U$ of $b_2(s)$ in $D_2$ which induces the isomorphism of \cref{fiber map}. Furthermore, exploiting the universality of the family $\cF_1$ at $b_1(s)$ in conjunction with \cref{eta=eta' coro}, we see that $\iota_1 = \psi \circ \iota_2$ holds on an open neighborhood of $s$ in $b_2^{-1}(U) \subset S$. Thus locally at the point $b_2(s)$ we have achieved what we wanted. If we are able to glue the pairs $(U,\psi)$ over an open neighborhood $W$ of $b_2(S)$ in $D_2$, we are done.

Unfortunately though there is a catch in the gluing: Given two pairs $(U,\psi)$ and $(U',\psi')$ as in the previous paragraph, the morphisms $\psi$ and $\psi'$ need not agree on the overlap $U\cap U'$. As a matter of fact, the universality of the family $\cF_1$ and \cref{eta=eta' coro} imply only that $\psi$ and $\psi'$ agree on an open subset $V$ of $U\cap U'$ containing all points $b_2(s)$ with $s \in b_2^{-1}(U\cap U')$. If $V \ne U\cap U'$, then $\psi$ and $\psi'$ will simply not glue to a morphism $(\cF_2)_{U\cup U'} \to \cF_1$.

The following \namecref{partial gluing} shows a way out of this predicament. The trick is to first pass from the open cover $\set{U,U'}$ of $U\cup U'$ to a suitably refined open cover---namely, to a so-called \emph{shrinking} of $\set{U,U'}$. The restrictions of $\psi$ and $\psi'$ to the refined open cover will then glue over an open set that contains all points of $b_2(S) \cap (U\cup U')$, which suffices for our purposes. \Cref{partial gluing} is inspired by an argument of Kashiwara's and Schapira's \cite[102--103]{KS94}.

\begin{lemm}\label{partial gluing}
Let $X$ be a topological space, $(U_i)_{i\in I}$ and $(V_i)_{i\in I}$ be indexed open covers of $X$, $\sF$ be a sheaf of sets on $X$, and $(s_i)_{i \in I}$ be an indexed family of sections $s_i \in \sF(U_i)$. Assume that the family of closed subsets $(\closure{V_i})_{i\in I}$ of $X$ is locally finite and satisfies $\closure{V_i} \subset U_{i}$ for all $i \in I$. Define
\[
A \coloneqq \setb*{x\in X}{\forall i,j \in I:x\in U_i\cap U_j \implies [s_i]_x = [s_j]_x}
\]
where $[\,\cdot\,]_x$ refers to taking the germ at $x$ in the sheaf $\sF$.
Then there exists an open subset $W$ of $X$ together with a section $t \in \sF(W)$ such that $A\subset W$ and
\[
\rest t{W\cap V_i} = \rest{s_{i}}{W\cap V_i}
\]
for all $i \in I$.
\end{lemm}
\begin{proof}
For every point $x \in X$ define $I(x) \coloneqq \setb{k \in I}{x \in \closure{V_k}}$. Further define
\[
W \coloneqq \setb*{x \in X}{\forall i,j \in I(x): [s_{i}]_x = [s_{j}]_x}.
\]
We contend that $W$ is an open subset of $X$ with $A \subset W$. Indeed, take $x \in A$. Then $x \in X$. For all $i,j \in I(x)$ we have that $x \in \closure{V_i} \cap \closure{V_j}$, whence $x \in U_{i}\cap U_{j}$ so that $[s_{i}]_x = [s_{j}]_x$ by the definition of $A$. This implies that $x \in W$.

Now fix a point $x \in W$. Since $(\closure{V_i})_{i\in I}$ is a locally finite indexed family of subsets of $X$, there exists an open subset $N\subset X$ containing $x$ such that the set $J \coloneqq \setb{i \in I}{\exists z \in \closure{V_i}\cap N}$ is finite. Note that $I(y) \subset J$ for all $y \in N$; in particular $I(x)$ is finite. Therefore---use induction on the cardinality of $I(x)$---there exists an open neighborhood $U' \subset \bigcap_{i \in I(x)}U_{i}$ of $x$ such that $\rest{s_{i}}{U'} = \rest{s_{j}}{U'}$ for all $i,j \in I(x)$. Defining $M \coloneqq N \setminus \bigcup_{j \in J\setminus I(x)}\closure{V_j}$, we observe that $M$ is an open subset of $X$ with $x \in M$. Moreover for every $y \in M$ we have that $I(y) \subset I(x)$. Thus, $M \cap U'$ is an open neighborhood of $x$ in $X$ such that $[s_{i}]_y = [s_{j}]_y$ for all $y \in M\cap U'$ and all $i,j \in I(y)$. In consequence $M\cap U' \subset W$, and since $x \in W$ was arbitrary, we see that $W$ is an open subset of $X$ as claimed.

Since $(V_i)_{i \in I}$ is an open cover of $X$, it is clear that $(W_i)_{i\in I}$ with $W_i = W\cap V_i$ is an open cover of $W$. Define the indexed family $(t_i)_{i\in I}$ by $t_i = \rest{s_{i}}{W_i}$. Let $i,j \in I$ be arbitrary indices and let $y \in W_i\cap W_j$ be a point. Then $y \in \closure{V_i}$ and $y \in \closure{V_j}$ so that $i,j \in I(y)$. Moreover, $y\in W$ so that $[s_{i}]_y = [s_{j}]_y$ by the definition of $W$. Since $[t_i]_y = [s_{i}]_y$, and likewise for $j$, we obtain that $[t_i]_y = [t_j]_y$. Given that $\sF$ is a sheaf of sets on $X$, it satisfies the locality sheaf axiom and we deduce that $\rest{t_i}{W_i\cap W_j} = \rest{t_j}{W_i\cap W_j}$. Employing the gluing sheaf axiom for $\sF$, this implies the existence of a section $t \in \sF(W)$ such that $\rest t{W_i} = t_i$ for all $i \in I$.
\end{proof}

\subsection{Proof of \cref*{universal map}}\label{universal map pf}
We proceed in steps, following the strategy mapped out in \cref{universal map idea}.

\subsubsection{Preparations}\label{universal map pf-1}
We define a presheaf of sets $\sH$ on $D_2$ by means of the following rules: $\sH(U)$ is, for every open subset $U$ of $D_2$, the set of morphisms of families from $(\cF_2)_U$ to $\cF_1$---that is,
\[
\sH(U) = \setb\psi{\psi \colon (\cF_2)_U \to \cF_1}.
\]
For every two open subsets $U$ and $V$ of $D_2$ with $V\subset U$ the restriction map $\rho^U_V\colon \sH(U) \to \sH(V)$ of $\sH$ satisfies $\rho^{U}_{V}(\phi) = \rest\phi V = \phi \circ \eta$ where $\eta \colon (\cF_2)_{V} \to (\cF_2)_U$ denotes the canonical injection of families. Observe that $\sH$ is not only a presheaf of sets on $D_2$ but a sheaf of sets on $D_2$.

\newcommand{\indphi}[2]{\iota_{#1,#2}}	% special command for induced morphisms
We define $I$ to be the set of all pairs $(U,\psi)$ where $U$ is an open subset of $D_2$ and $\psi \in \sH(U)$ is an element such that $\indphi1U = \psi \circ \indphi2{U}$ where
\[
\indphi1{U} = \rest{\iota_1}{b_2^{-1}(U)} \colon \cF_{b_2^{-1}(U)} \to \cF_1 \qquad \text{and} \qquad \indphi2{U} \colon \cF_{b_2^{-1}(U)} \to (\cF_2)_U
\]
denote the morphisms of families induced by $\iota_1$ and $\iota_2$, respectively. The notation $\indphi1{U}$ must not be confused with the notation $(\iota_1)_U$ of \cref{restriction sec}.

By the semicontinuity theorem \cite{Rie70}, since the family of compact complex manifolds $\cF_2$ is given by a proper holomorphic submersion, the function
\[
d_0 \colon D_2 \to \NN, \qquad d_0(y) = \hdim0{\cF_2(y)}{\Theta_{\cF_2(y)}},
\]
is upper semicontinuous. Thus according to the assumptions in \cref{universal map} there exists an open neighborhood $N$ of $b_2(S)$ in $D_2$ such that $d_0(y) = 0$ for all $y \in N$. Without loss of generality we assume that $N = D_2$.

\subsubsection{Proof of the uniqueness}\label{uniqueness of ext}
Let $(U,\psi)$ and $(U',\psi')$ be two elements of $I$. Then $\indphi1{U} = \psi \circ \indphi2{U}$ and likewise $\indphi1{U'} = \psi' \circ \indphi2{U'}$. Let $s \in b_2^{-1}(U\cap U')$ be an arbitrary point. Then
\[
\psi(b_2(s)) = \psi'(b_2(s)) \colon \cF_2(b_2(s)) \to \cF_1(b_1(s)).
\]
Moreover, by \cref{universality intro} and the assumptions of \cref{universal map}, the family $\cF_1$ is universal at the point $b_1(s)$. Therefore, when $\psi$ and $\psi'$ are morphisms of families over $h$ and $h'$, respectively, we see that $h$ and $h'$ agree on an open neighborhood of $b_2(s)$ in $U\cap U'$. Since $s$ was arbitrary, we deduce the existence of an open subset $V \subset U\cap U'$ such that $\rest h{V} = \rest{h'}{V}$ and $b_2(b_2^{-1}(U\cap U')) \subset V$. Now applying \cref{eta=eta' coro} to the restrictions $\rest\psi V$ and $\rest{\psi'}V$, we deduce the existence of an open subset $W$ of $V$ such that $b_2(b_2^{-1}(U\cap U')) \subset W$ and $\rest\psi W = \rest{\psi'}W$.

When $b_2(S) \subset U$ and $b_2(S) \subset U'$, then $b_2^{-1}(U\cap U') = S$ so that $b_2(S) \subset W$. Thus the preceding argument proves \cref{universal un} of \cref{universal map}.

\subsubsection{Existence at points}\label{ext at points}
Let $s \in S$ be an arbitrary point. We contend the existence of a pair $(V,\phi) \in I$ such that $b_2(s) \in V$.

Indeed, since the family $\cF_1$ is complete at $b_1(s)$, there exists an open subset $U$ of $D_2$ containing $b_2(s)$ as well as a morphism of families $\psi \colon (\cF_2)_U \to \cF_1$ so that \[\iota_1(s) = \psi(b_2(s)) \circ \iota_2(s).\] Applying \cref{eta=eta' coro} in conjunction with the fact that $\cF_1$ is universal at $b_1(s)$, we see there exists an open subset $T \subset b_2^{-1}(U)$ such that $s \in T$ and
\[
\rest{\iota_1}T = \rest{\indphi1{U}}T = \rest{\psi \circ \indphi2{U}}T \colon \cF_T \to \cF_1.
\]
As the holomorphic map $b_2 \colon S\to D_2$ is an embedding by assumption, $b_2$ induces a homeomorphism between $S$ and the subspace $b_2(S)$ of $D_2$. This implies that $T = b_2^{-1}(V)$ for an open subset $V$ of $U$. Hence $\indphi1{V} = \phi \circ \indphi2{V}$ where $\phi \coloneqq \rest\psi V$, which proves our claim.

\subsubsection{Gluing and global existence}\label{shrinking}
Define $(U_i)_{i\in I}$ and $(\psi_i)_{i\in I}$ to be the indexed families given by the assignments $(U,\psi) \mapsto U$ and $(U,\psi) \mapsto \psi$, respectively, and define $Y$ to be the open subspace of $D_2$ induced on the union $\bigcup_{i\in I} U_i$. Then $(U_i)_{i\in I}$ is an indexed open cover of $Y$. Being a second-countable complex space, $Y$ is regular and Lindelöf, whence paracompact \cite[][Theorem 41.5]{Munkres}. Thus there exists a locally finite indexed open cover $(V_i)_{i\in I}$ of $Y$ such that $\closure{V_i} \subset U_i$ for all $i \in I$, where the closure is taken in $Y$ \cite[][Lemma 41.6]{Munkres}.

By \cref{ext at points} we know that $b_2(S) \subset Y$. Let $s \in S$ be an arbitrary point and $i,j \in I$ be elements such that $b_2(s) \in U_i\cap U_j$. Then according to \cref{uniqueness of ext} there exists an open subset $U' \subset U_i\cap U_j$ such that $b_2(s) \in U'$ and $\rest{\psi_i}{U'} = \rest{\psi_j}{U'}$. In other words, the germs of $(U_i,\psi_i)$ and $(U_j,\psi_j)$ at $b_2(s)$ in the sheaf $\sH$, equivalently in the sheaf $\rest \sH Y$, agree. Therefore \cref{partial gluing} implies the existence of an open subset $W \subset Y$ and an element $\phi \in \sH(W)$ such that $b_2(S) \subset W$ and
\[
\rest\phi{W\cap V_i} = \rest{\psi_i}{W\cap V_i}
\]
for all $i \in I$.

\subsubsection{Conclusion}
We contend that $\iota_1 = \phi \circ \indphi2{W}$, which proves \cref{universal ex} of \cref{universal map} and is actually equivalent to saying that $(W,\phi) \in I$. We note that ${\indphi1{U_i} = \psi_i \circ \indphi2{U_i}}$ for all $i \in I$. Thus
\begin{align*}
\rest{\iota_1}{b_2^{-1}(W\cap V_i)} &= \rest{\indphi1{U_i}}{b_2^{-1}(W\cap V_i)} = \rest{\paren*{\psi_i \circ \indphi2{U_i}}}{b_2^{-1}(W\cap V_i)} = \rest{\psi_i}{W\cap V_i} \circ \indphi2{W\cap V_i}\\ &= \rest\phi{W\cap V_i} \circ \indphi2{W\cap V_i} = \rest{(\phi \circ \indphi2W)}{b_2^{-1}(W\cap V_i)}
\end{align*}
for all $i\in I$.
In addition, since $b_2(S) \subset W$, since $(V_i)_{i\in I}$ is an indexed open cover of $Y$, and since $W\subset Y$, we see that
\[
S = b_2^{-1}(W) = b_2^{-1}\paren*{\bigcup_{i\in I}(W\cap V_i)} = \bigcup_{i\in I} b_2^{-1}(W\cap V_i)
\]
and our claim follows.\qed

\begin{rema}\label{torellilike}
Let $(\cF',\mu')$ be a $\Lambda$-marked family of ihs manifolds over $S'$ and $s\in S'$ a point. Then $\cF'$ is semi-universal at $s$ if and only if the period map $p' \colon S' \to \pdom\Lambda$ of $(\cF',\mu')$ is a local biholomorphism at $s$.
\end{rema}
\begin{proof}
By \cref{local Torelli} there exists a semi-universal deformation $(\cX,(i,j))$ of the fiber $\cF'(s)$ where $\cX$ is a family of ihs manifolds over a simply connected complex manifold $S$. Thus there is an open subspace $U \subset S'$ and a morphism of families $(g,h) \colon \cF'_U \to \cX$ such that $s \in U$ and $h(s) = j(0)$. Using \cref{marking exists,induced marking}, we can assume that $(g,h) \colon (\cF',\mu')_U \to (\cX,\mu)$ is a morphism of $\Lambda$-marked families for a $\Lambda$-marking $\mu$ of $\cX$. As a consequence, if $p \colon S \to \pdom\Lambda$ denotes the period map of $(\cX,\mu)$, we know that $\rest{p'}U = p \circ h$ by \cref{pm functorial}.

By \cref{local Torelli} we see that $p$ is a local biholomorphism at $j(0)$. Therefore $p'$ is a local biholomorphism at $s$ if and only if $h \colon U \to S$ is a local biholomorphism at $s$. The latter is clearly equivalent to $\cF'$ being semi-universal at $s$.
\end{proof}

\Cref{torellilike} allows for a marked family version of \cref{universal map} which we need in \cref{extension sec}. We formulate only the existence part---that is, \cref{universal ex}---as the uniqueness part would only be weaker than that of \cref{universal map}.

\begin{coro}\label{universal map ihs}
Let $\Lambda$ be a lattice, $\cF$, $\cF_1$, and $\cF_2$ be $\Lambda$-marked families of ihs manifolds over $S$, $D_1$, and $D_2$, respectively, and $\iota_k \colon \cF \to \cF_k$ be morphisms of $\Lambda$-marked families over $b_k$ for $k \in \set{1,2}$. Assume that the period map of $\cF_1$ is a local biholomorphism $D_1 \to \pdom\Lambda$ and that $b_2 \colon S\to D_2$ is an embedding into a second-countable space $D_2$. Then there exists an open subset $W \subset D_2$ as well as a morphism of $\Lambda$-marked families $\phi \colon (\cF_2)_W \to \cF_1$ such that $b_2(S) \subset W$ and $\iota_1 = \phi \circ \iota_2$.
\end{coro}
\begin{proof}
Denote by $\cF'$, $\cF_1'$, and $\cF_2'$ the families underlying the marked families $\cF$, $\cF_1$, and $\cF_2$, respectively.
Since the period map of $\cF_1$ is a local biholomorphism, the family of compact complex manifolds $\cF_1'$ is semi-universal at every point $y \in D_1$ by \cref{torellilike}. Thus \cref{universal map} implies the existence of an open subset $\tilde{W} \subset D_2$ and a morphism of families $\tilde\phi \colon (\cF_2')_{\tilde W} \to \cF_1'$ such that $\iota_1 = \tilde\phi \circ \iota_2$.

By \cref{induced marking} we know there exists a unique marking $\nu$ of the family of ihs manifolds $(\cF_2')_{\tilde W}$ for which $\tilde{\phi}$ becomes a morphism of marked families between $((\cF_2')_{\tilde W},\nu)$ and $\cF_1$. Since $\iota_1 = \tilde\phi \circ \iota_2$ and since $\iota_k \colon \cF \to \cF_k$ are morphisms of marked families for $k \in \set{1,2}$, we see that $\nu$ and the marking of $\cF_2$ induce the same marking on the fiber $\cF_2'(b_2(s))$ for all $s \in S$. Define $W$ to be the union of all connected components $C$ of $\tilde W$ for which there exists a point in $b_2(S) \cap C$. Then $W$ is open in $D_2$, we know that $b_2(S) \subset W$, and by \cref{markings coincide} the restriction $\phi \coloneqq \rest{\tilde{\phi}}{W} \colon (\cF_2)_W \to \cF_1$ is a morphism of marked families with $\iota_1 = \phi \circ \iota_2$.
\end{proof}

\begin{rema}\label{markings coincide}
Let $\Lambda$ be a lattice, $\cF$ be a family of ihs manifolds over a connected space $S$, and $\mu$ and $\nu$ be two $\Lambda$-markings of $\cF$. Let $s \in S$ be a point and assume that $\mu$ and $\nu$ induce the same marking
\[
\mu_s = \nu_s \colon \Hsh2{\cF(s)}\ZZ \to \Lambda
\]
on the fiber $\cF(s)$ in the sense of \cref{ihs family sec}. Then $\mu = \nu$.

Indeed, writing $\cF = (X,S,f)$, the sheaf of abelian groups $\Rsh2f{\csh\ZZ X}$ is constant, for it is isomorphic to $\csh\Lambda S$ by virtue of $\mu$ (or $\nu$). Thus, as the space $S$ is connected, the canonical map
\[
\paren*{\Rsh2f{\csh\ZZ X}}(S) \to \paren*{\Rsh2f{\csh\ZZ X}}_s
\]
from global sections to the stalk at $s$ is bijective. In consequence the global section components of the morphisms of sheaves $\mu$ and $\nu$ coincide. Employing the same argument again, we conclude that $\mu_t = \nu_t$ for all $t \in S$, which implies our claim.
\end{rema}

\section{Proof of the extension theorem}
\label{extension sec}

\subsection{}\label{extension intro}
In what follows we prove \cref{ihs extension thm}.
We assume throughout \cref{extension sec} that a lattice $\Lambda$ and a $\Lambda$-marked family of ihs manifolds $\cF$ over a complex space $S$ are given such that the period map $h \colon S\to \pdom\Lambda$ of $\cF$ is an embedding. In order to simplify the notation, we assume, without loss of generality, that $S \subset \pdom\Lambda$ is a complex subspace and $h$ is the corresponding canonical injection. That way we need not distinguish between the points $s$ and $h(s)$ for $s \in S$.

The construction of the extension $\tilde\cF$ of the marked family $\cF$ somewhat parallels the proof of \cref{universal map}, compare \cref{universal map idea}: first we produce suitable extensions locally at every point of $S$, then we explain how to glue the local extensions to a global one. The local considerations are dealt with in \cref{step 1,step 2,step 3}. \Cref{gluing families} provides a general gluing device for $\Lambda$-marked families of ihs manifolds. The final conclusions are drawn in \cref{actual proof}.

\begin{defi}
We say that $(\cK,\iota)$ is an \emph{admissible extension} over $U$ when $U \subset \pdom\Lambda$ is an open subspace, $\cK$ is $\Lambda$-marked family of ihs manifolds whose period map is the canonical injection $U \inj \pdom\Lambda$, and $\iota \colon \cF_{S\cap U} \to \cK$ is a morphism of $\Lambda$-marked families over the canonical injection $S\cap U \inj U$.
\end{defi}

\begin{prop}\label{step 1}
Let $s \in S$ be a point. Then there exists an admissible extension $(\cK,\iota)$ over $U$ such that $s \in U$.
\end{prop}
\begin{proof}
We can write $\cF = (\cY,\nu)$. By \cref{local Torelli} there exists a semi-universal deformation $(\cX,(i,j))$ of the fiber $\cY(s)$ such that $\cX$ is a family of ihs manifolds over a simply connected complex manifold $U$. Due to the universality we obtain a morphism of families $(a,b) \colon \cY_V \to \cX$ such that $V \subset S$ is an open subspace with $s \in V$ and $b(s) = j(0)$. Just like in the proof of \cref{torellilike} we can assume that $(a,b) \colon \cF_V \to (\cX,\mu)$ is a morphism of $\Lambda$-marked families for a $\Lambda$-marking $\mu$ of $\cX$. Since the period map $g \colon U \to \pdom\Lambda$ of $(\cX,\mu)$ is a local biholomorphism at $j(0)$, we can further assume that $U \subset \pdom\Lambda$ is an open subspace and $g$ is the associated canonical injection.
Now \cref{pm functorial} tells us that $\rest{h}V = g \circ b$. Recalling from \cref{extension intro} that $h \colon S \inj \pdom\Lambda$ is the canonical injection of a complex subspace, we conclude that $b \colon V \inj U$ is the canonical injection of a complex subspace, too. Finally we can replace $U$ by a smaller open subspace of $\pdom\Lambda$ so that $V = S \cap U$. Hence $((\cX,\mu),(a,b))$ is an admissible extension over $U$.
\end{proof}

When $(U_i)_{i \in I}$ is an indexed family of open subsets of a given topological space and $i,j,k \in I$ are indices, we emloy the standard notation $U_{ij} \coloneqq U_i\cap U_j$ and $U_{ijk} \coloneqq U_i\cap U_j\cap U_k$ for a double and a triple intersection, respectively.

\begin{prop}\label{step 2}
Let $(\cK_1,\iota_1)$ and $(\cK_2,\iota_2)$ be admissible extensions over $U_1$ and $U_2$, respectively. Then there exists an open subspace $W \subset V \coloneqq U_{12}$ together with a $W$-morphism of $\Lambda$-marked families
\[
\phi \colon (\cK_2)_W \to (\cK_1)_W
\]
such that $S \cap V \subset W$ and $\phi \circ \bar\iota_2 = \bar\iota_1$, where $\bar\iota_k \colon \cF_{S\cap V} \to (\cK_k)_W$ denotes the morphism induced by $\iota_k$.
\end{prop}
\begin{proof}
We apply \cref{universal map ihs} to the $\Lambda$-marked families $\cF_{S\cap V}$, $(\cK_1)_{V}$, and $(\cK_2)_{V}$ and the morphisms of $\Lambda$-marked families $\bar\iota_k \colon \cF_{S\cap V} \to (\cK_k)_{V}$ induced by $\iota_k$ for $k \in \set{1,2}$.
The assumptions of \cref{universal map ihs} are clearly fulfilled, so we obtain an open subspace $W\subset V$ and a morphism of $\Lambda$-marked families $\phi' \colon (\cK_2)_W \to (\cK_1)_{V}$ with $S\cap V\subset W$ and $\bar\iota_1 = \phi' \circ \bar\iota_2$.
As the period maps of $(\cK_1)_W$ and $(\cK_2)_{V}$ are the canonical injections $W \inj \pdom\Lambda$ and $V \inj \pdom\Lambda$, respectively, \cref{pm functorial} implies that $\phi'$ is a morphism over the canonical injection $W \inj V$. Thus $\phi'$ induces the desired morphism of marked families $\phi$.
\end{proof}

\begin{prop}\label{step 3}
For $k \in \set{1,2,3}$ let $(\cK_k,\iota_k)$ be an admissible extension over $U_k$. Moreover, for $i,j \in \set{1,2,3}$ with $i<j$, let $W_{ij} \subset U_{ij}$ be an open subspace and
\[
\phi_{ij} \colon (\cK_j)_{W_{ij}} \to (\cK_i)_{W_{ij}}
\]
be a $W_{ij}$-morphism of $\Lambda$-marked families such that $S \cap U_{ij} \subset W_{ij}$ and ${\bar\iota_i = \phi_{ij} \circ \bar\iota_j}$ for the induced morphisms. Then ${W\defeq W_{12} \cap W_{13} \cap W_{23}}$ contains an open subspace $Z$ such that
\[
(\phi_{13})_Z = (\phi_{12})_Z \circ (\phi_{23})_Z
\]
and $S \cap U_{123} \subset Z$.
\end{prop}
\begin{proof}
For $k \in \set{1,2,3}$ let $\bar\iota_k \colon \cF_{S\cap W} \to (\cK_k)_W$ denote the morphism of $\Lambda$-marked families that is induced by $\iota_k$. Then by assumption
\[
\bar\iota_1 = (\phi_{12})_W \circ \bar\iota_2 = (\phi_{12})_W \circ (\phi_{23})_W \circ \bar\iota_3 \quad \text{and} \quad \bar\iota_1 = (\phi_{13})_W \circ \bar\iota_3 .
\]
Thus the claim follows immediately from \cref{universal un} of \cref{universal map} if we notice that $S\cap U_{123} \subset S\cap W$; in fact, the latter two sets are equal.
\end{proof}

\begin{lemm}\label{gluing families}
Let $D$ be a complex space and $(\cF_i)_{i\in I}$ be an indexed family such that $\cF_i$ is a $\Lambda$-marked family of ihs manifolds over an open subspace $U_i \subset D$ for every $i \in I$. Moreover let $(\phi_{ij})_{i,j \in I}$ be an indexed family such that, for all $i,j,k \in I$, firstly,
\[
\phi_{ij} \colon (\cF_j)_{U_{ij}} \to (\cF_i)_{U_{ij}}
\]
is a $U_{ij}$-morphism of $\Lambda$-marked families and, secondly,
\[
(\phi_{ik})_{U_{ijk}} = (\phi_{ij})_{U_{ijk}} \circ (\phi_{jk})_{U_{ijk}}.
\]

Then there exists a $\Lambda$-marked family of ihs manifolds $\tilde\cF$ over the open subspace $\tilde U \coloneqq \bigcup_{i \in I}U_i$ of $D$ together with an indexed family $(\eta_i)_{i \in I}$ such that, for all $i,j \in I$,
\[
\eta_i \colon \cF_i \to \tilde\cF_{U_i}
\]
is a $U_i$-morphism of $\Lambda$-marked families and
\[
(\eta_i)_{U_{ij}} \circ \phi_{ij} = (\eta_j)_{U_{ij}}.
\]
\end{lemm}
\begin{proof}
Let us write $\cF_i$ as $(X_i,U_i,f_i,\mu_i)$ and $\phi_{ij}$ as $(g_{ij},\id{U_{ij}})$ for all $i,j \in I$. Then the quadruple
\[
\paren*{I,(X_i)_{i\in I},(f_i^{-1}(U_{ij}))_{i,j \in I},(g_{ij})_{i,j \in I}}
\]
is Hausdorff gluing data for a complex space in the sense of Fischer \cite[\nopp 0.24]{Fischer}.
Thus we obtain a complex space $\tilde X$ together with an indexed family $(a_i)_{i \in I}$ such that $a_i \colon X_i \to \tilde X$ is an open embedding and
\[
\rest{a_i}{f_i^{-1}(U_{ij})} \circ g_{ij} = \rest{a_j}{f_j^{-1}(U_{ij})}
\]
for all $i,j \in I$. In fact we can take the underlying topological space of $\tilde X$ to be the quotient of the disjoint union $\bigsqcup_{i\in I}X_i$ by the equivalence relation under which $(x,i) \sim (y,j)$ if and only if $x = g_{ij}(y)$; then $a_i$ is given by $a_i(x) = [(x,i)]$. Since for all $i,j \in I$
\[
\rest{f_i}{f_i^{-1}(U_{ij})} \circ g_{ij} = \rest{f_j}{f_j^{-1}(U_{ij})},
\]
there exists a unique holomorphic map $\tilde f \colon \tilde X \to \tilde U$ such that $\tilde f \circ a_i = b_i \circ f_i$ for all $i \in I$ where $b_i \colon U_i \to \tilde U$ denotes the canonical injection. Therefore the triple $(\tilde X,\tilde U,\tilde f)$ is a family of ihs manifolds and, for all $i \in I$, the pair $(a_i,b_i)$ is a morphism of families between $(X_i,U_i,f_i)$ and $(\tilde X,\tilde U,\tilde f)$.

Defining $(\tilde X_i,U_i,\tilde f_i) \coloneqq (\tilde X,\tilde U,\tilde f)_{U_i}$ for $i \in I$, there exists a unique morphism of sheaves of abelian groups
\[
\tilde\mu_i \colon \Rsh2{(\tilde f_i)}{\csh\ZZ{\tilde X_i}} \to \csh\Lambda{U_i}
\]
such that $(a_i,b_i)$ induces a $U_i$-morphism of $\Lambda$-marked families of ihs manifolds
\[
\eta_i \colon \cF_i = (X_i,U_i,f_i,\mu_i) \to (\tilde X_i,U_i,\tilde f_i,\tilde\mu_i).
\]
We regard $\tilde\mu_i$ as a morphism of sheaves of abelian groups
\[
\tilde\mu_i \colon \rest{\Rsh2{\tilde f}{\csh\ZZ{\tilde X}}}{U_i} \to \rest{\csh\Lambda{\tilde U}}{U_i}
\]
and notice that $\rest{\tilde\mu_i}{U_{ij}} = \rest{\tilde\mu_j}{U_{ij}}$ for all $i,j \in I$ since $\phi_{ij}$ is a morphism of $\Lambda$-marked families.
Hence there exists a unique morphism of sheaves of abelian groups
\[
\tilde\mu \colon \Rsh2{\tilde f}{\csh\ZZ{\tilde X}} \to \csh\Lambda{\tilde U}
\]
such that $\rest{\tilde\mu}{U_i} = \tilde\mu_i$ for all $i \in I$. As a consequence $\tilde\cF \coloneqq (\tilde X,\tilde U,\tilde f,\tilde\mu)$ is a $\Lambda$-marked family of ihs manifolds with the property that $\eta_i \colon \cF_i \to \tilde\cF_{U_i}$ is a $U_i$-morphism of $\Lambda$-marked families for every $i \in I$.
\end{proof}

\subsection{Proof of \cref*{ihs extension thm}}\label{actual proof}
We proceed in three steps.

\subsubsection{Local extensions}
From \cref{step 1} we deduce the existence of indexed families $(U_i)_{i\in I}$ and $\paren*{(\cK_i,\iota_i)}_{i\in I}$ such that, firstly, $(\cK_i,\iota_i)$ is an admissible extension over $U_i$ for all $i \in I$ and, secondly, $S$ is contained in the open subspace $D \coloneqq \bigcup_{i\in I}U_i$ of $\pdom\Lambda$. In fact we can take $I$ equal to the set of points of $S$ and stipulate that $s \in U_s$ for all $s \in S$.
Employing \cref{step 2} we deduce the existence of an indexed family $\paren*{(W_{ij},\phi_{ij})}_{i,j \in I}$ such that $W_{ij} \subset U_{ij}$ is an open subspace with $S\cap U_{ij} \subset W_{ij}$ and
\[
\phi_{ij} \colon (\cK_j)_{W_{ij}} \to (\cK_i)_{W_{ij}}
\]
is a $W_{ij}$-morphism of $\Lambda$-marked families satisfying $\bar\iota_i = \phi_{ij} \circ \bar\iota_j$ for all $i,j \in I$.
Employing \cref{step 3} we deduce the existence of an indexed family $(Z_{ijk})_{i,j,k \in I}$ such that $Z_{ijk} \subset W_{ij}\cap W_{ik}\cap W_{jk}$ is an open subspace with
\[
(\phi_{ik})_{Z_{ijk}} = (\phi_{ij})_{Z_{ijk}} \circ (\phi_{jk})_{Z_{ijk}}
\]
and $S \cap U_{ijk} \subset Z_{ijk}$.

\subsubsection{Shrinking}
Like in \cref{shrinking} there exists an indexed open cover $(V_i)_{i\in I}$ of $D$ such that the family $(\closure{V_i})_{i\in I}$ of closed subsets of $D$ is locally finite and satisfies $\closure{V_i} \subset U_i$ for all $i \in I$.
For every $x\in D$ define $I(x) \defeq \setb{i\in I}{x \in \closure{V_i}}$ and consider set
\[
W \coloneqq \setb*{x \in D}{\forall i,j \in I(x): x\in W_{ij} \text{ and } \forall i,j,k\in I(x): x\in Z_{ijk}}.
\]
Then in analogy to the proof of \cref{partial gluing} we verify that
\begin{enumerate}
\item\label{prop1} $W$ is open in $D$,
\item\label{prop2} $S \subset W$,
\item\label{prop3} $(V_i\cap W)\cap (V_j\cap W) \subset W_{ij}$ for all $i,j \in I$, and
\item\label{prop4} $(V_i\cap W)\cap (V_j\cap W)\cap (V_k\cap W) \subset Z_{ijk}$ for all $i,j,k \in I$.
\end{enumerate}
As $(\closure{V_i})_{i\in I}$ is locally finite, the set $W$ is locally the intersection of finitely many of the open subspaces $W_{ij}$ and $Z_{ijk}$, which implies \cref{prop1}.
Let $s\in S$ be a point and $i,j\in I(s)$. Then $s\in\closure{V_i}\cap\closure{V_j}\subset U_{ij}$ and $s\in S\cap U_{ij}\subset W_{ij}$. If moreover $k\in I(x)$, then $s\in U_{ijk}$ and $s\in S\cap U_{ijk}\subset Z_{ijk}$. Thus $s\in W$, which proves \cref{prop2}.
If $x\in V_{ij}\cap W$, then $i,j\in I(x)$; and if $x\in V_{ijk}\cap W$, then $i,j,k\in I(x)$. So \cref{prop3,prop4} follow immediately from the definition of $W$.

\subsubsection{Gluing}
Let the indexed families $(V'_i)_{i\in I}$, $(\cF_i)_{i\in I}$, and $(\phi'_{ij})_{i,j\in I}$ be given by
\[
V'_i = V_i \cap W, \quad
\cF_i = (\cK_i)_{V'_i}, \quad \text{and} \quad \phi'_{ij} = (\phi_{ij})_{V'_{ij}},
\]
respectively. Then applying \cref{gluing families}, we obtain a $\Lambda$-marked family of ihs manifolds $\tilde\cF$ over the open subspace $W = \bigcup_{i\in I}V'_i$ of $D$ as well as an indexed family $(\eta_i)_{i\in I}$ of $V'_i$-morphisms of $\Lambda$-marked families $\eta_i \colon \cF_i \to \tilde\cF_{V'_i}$ such that
\[
(\eta_i)_{V'_{ij}} \circ \phi'_{ij} = (\eta_j)_{V'_{ij}}
\]
for all $i,j \in I$.

Writing $\bar\iota_k \colon \cF_{S\cap V'_k} \to \cF_k$ for the morphism that is induced by $\iota_k$, we conclude that the compositions $\eta_i \circ \bar\iota_i$ and $\eta_j \circ \bar\iota_j$ agree on the intersection $V'_{ij}$ for all $i,j \in I$. Hence, as $S$ is contained in the union $\bigcup_{i\in I}V'_i$, there exists a unique morphism of $\Lambda$-marked families $\eta \colon \cF \to \tilde\cF$ over $h \colon S\inj W \subset \pdom\Lambda$ such that $\eta$ induces $\eta_i \circ \bar\iota_i$ on $V'_i$ for all $i \in I$. Last but not least, the period map of $\tilde\cF$ is the canonical injection $W\inj \pdom\Lambda$ simply because, for all $i \in I$, the period map of $\cF_i$ is the canonical injection $V'_i \inj \pdom\Lambda$. \qed

\printbibliography
%\bibliographystyle{plain}
%\bibliography{unobs}
\end{document}